\documentclass[reqno]{amsart}

\usepackage{mystyle}

\usepackage{microtype}
\usepackage{subfigure}
\usepackage{booktabs}
\usepackage[hidelinks,breaklinks]{hyperref}

\usepackage{amsmath}
\usepackage{amssymb}
\usepackage{mathtools}
\usepackage[foot]{amsaddr}
\usepackage{cite}
\usepackage{textcomp}

\usepackage[capitalize,noabbrev]{cleveref}

\theoremstyle{plain}
\newtheorem{theorem}{Theorem}[section]
\newtheorem{proposition}[theorem]{Proposition}
\newtheorem{lemma}[theorem]{Lemma}

\theoremstyle{definition}
\newtheorem{definition}[theorem]{Definition}

\theoremstyle{remark}
\newtheorem{remark}[theorem]{Remark}

\makeatletter
\def\paragraph{\medskip\@startsection{paragraph}{4}%
  \z@\z@{-\fontdimen2\font}%
  {\normalfont\bfseries}}
\makeatother

\usepackage{autonum}

\title[Unified framework for convex optimization methods: numerical analysis approach]{A new unified framework\\ for designing convex optimization methods\\ with prescribed theoretical convergence estimates:\\ A numerical analysis approach}
\author{Kansei Ushiyama}
\email{ushiyama-kansei074@g.ecc.u-tokyo.ac.jp}
\author{Shun Sato}
\author{Takayasu Matsuo}
\address{Department of Mathematical Informatics, Graduate School of Information Science and Technology, The University of Tokyo, Tokyo, Japan}

\begin{document}

\begin{abstract}
We propose a new unified framework for describing and designing gradient-based convex optimization methods from a numerical analysis perspective.
There the key is the new concept of {\em weak discrete gradients} (weak DGs), which is a generalization of DGs standard in numerical analysis.
Via weak DG, we consider abstract optimization methods, and prove unified convergence rate estimates that hold independent of the choice of weak DGs except for some constants in the final estimate.
With some choices of weak DGs, we can reproduce many popular existing methods, such as the steepest descent and Nesterov's accelerated gradient method,
and also some recent variants from numerical analysis community.
By considering new weak DGs, we can easily explore new theoretically-guaranteed optimization methods; we show some examples. 
We believe this work is the first attempt to fully integrate research branches in optimization and numerical analysis areas, so far independently developed.
\end{abstract}


\maketitle

\section{Introduction}
\label{sec:intro}

This work aims at defining a new unified approach for designing optimization methods by integrating standard optimization theories and some theories from numerical analysis.
This section first summarizes an underlying big view.
Then we survey related literature, and clarify our main contributions.

\subsection{Preliminaries and notation}
\subsubsection{Basic concepts in optimization}

In this paper, we consider unconstrained optimization problems
\begin{equation}
\min_{ x \in \RR^d} f (x)
\end{equation}
on the $d$-dimensional Euclidean space $ \RR^d$ ($d$ is a positive integer) with the standard inner product $ \inpr{\cdot}{\cdot} $ and the induced norm $ \norm{\cdot }$, 
where $f \colon \RR^d \to \RR$ denotes a differentiable objective function.
We assume the existence of the optimal value $ f^{\star} $ and the optimal solution $ x^{\star}$. 

The following function classes are essential.
\begin{definition}
A differentiable function $f$ is said to be $L$-smooth if $ \nabla f $ is $L$-Lipschitz continuous. 
\end{definition}

\begin{definition}
A function $f$ is said to be $\mu$-strongly convex if $ f - \normm{\cdot}/2 $ is convex. 
\end{definition}

When $f$ is $\mu$-strongly convex and differentiable, 
\begin{equation}\label{scineq}
\frac{\mu}{2} \normm{y-x} \le f(y) - f(x) - \inpr{\nabla f(x)}{y-x}
\end{equation}
holds for any $ x, y \in \RR^d $. 

In this paper, we mainly consider convex objective functions (see~\cref{par:PL} for an exception).

\subsubsection{Continuous systems for optimization and their convergence rates}
\label{subsubsec:rates_conti}

For the unconstrained optimization problem, we can consider various optimization methods.
Here, however, let us consider continuous systems that can serve as optimizers.

To this end, it is essential to consider gradient systems (ordinary differential equations; ODEs), for example, the gradient flow
\begin{equation}\label{contgf}
    \dot{x} = - \nabla f(x), \quad x(0) = x_0 \in \RR^d. 
\end{equation}
Observe that it is driven by the gradient of the objective function.
As its mathematical consequence, the value of the objective function decreases:
\begin{align}
  \dt f(x(t)) = \inpr{\nabla f(x(t))}{\dot{x}(t)}
  = -\norm*{\nabla f(x(t))}^2 \le 0,
  \label{pf:gengf}
\end{align}
which means it can serve as an optimizer.
Notice that the proof relies on the {\em chain rule of differentiation}, and the {\em form of the system} itself~\eqref{contgf}; this will be extremely important in what follows.


Despite its simpleness, its convergence rate varies depending on the class of objective functions.
Below we show some known results.
The following rates are proved by the so-called Lyapunov argument, 
which introduces a ``Lyapunov function'' that explicitly contains the convergence rate.
The proof is left to \cref{app:proof_conti}, but we recommend the readers to have a glance at it (at least~\cref{subsec:proof_cr_conv}) and notice that it requires~the {\em convexity inequality}~\eqref{scineq}; this will be important as well.

%
\begin{theorem}[Convex case]\label{thm:cr_conv}
  Suppose that $f$ is convex. Let $x \colon [0,\infty) \to \RR^d$ be the solution of the gradient flow~\eqref{contgf}.
  Then the solution satisfies
  \begin{equation}
    f(x(t)) - f^\star \le \frac{\normm{x_0-x^\star}}{2t}.
  \end{equation}
\end{theorem}

\begin{theorem}[Strongly convex case]\label{thm:cr_sc}
  Suppose that $f$ is $\mu$-strongly convex. Let $x \colon [0,\infty) \to \RR^d$ be the solution of the gradient flow~\eqref{contgf}. 
  Then the solution satisfies
  \begin{equation}	
		f(x(t)) - f^\star \le \e{-\mu t}\normm{x_0-x^\star}.
	\end{equation}
\end{theorem}

Next, we consider ODEs corresponding to the accelerated gradient methods, including Nesterov's one.
As is well known, the forms of the accelerated gradient methods are different depending on the class of objective functions, and consequently the ODEs to be considered also change. 
In this paper, we call them \emph{accelerated gradient flows}. 

%

When the objective functions are convex, we consider the following ODE~\cite{WRJ21}:
let $A \colon \RR_{\ge 0} \to \RR_{\ge 0}$ be a differentiable strictly monotonically increasing function with $A(0)=0$, and
  \begin{equation} 
    \simulparen{
      \dot{x} &= \frac{\dot{A}}{A}(v-x),\\
      \dot{v} &= -\frac{\dot{A}}{4} \nabla f(x),
    }  \label{accgf}
  \end{equation}
  with $(x(0),v(0)) = (x_0,v_0) \in \RR^d \times \RR^d$. 

\begin{theorem}[Convex case~\cite{WRJ21}]\label{thm:cr_ac_conv}
  Suppose that $f$ is convex. 
  Let $(x,v) \colon [0,\infty) \to \RR^d \times \RR^d$ be the solution of the ODE~\eqref{accgf}.
  Then it satisfies
  \begin{equation}
    f(x(t)) - f^\star \le \frac{2\norm{x_0 - x^\star}^2}{A(t)}.
  \end{equation} 
\end{theorem}
\begin{remark}
If we set $A(t)=t^2$, this system coincides with a continuous limit ODE of the accelerated gradient method for convex functions
\begin{equation}
  \ddot{x} + \frac{3}{t}\dot{x} + \nabla f(x)=0,
\end{equation}
which is derived in~\cite{SBC16}.
\end{remark}

Next, for strongly convex objective functions, let us consider the ODE~\cite{WRJ21}:
  \begin{equation}
    \simulparen{
      \dot{x} &= \mur(v -x), \\
      \dot{v} &= \mur(x-v-\nabla f(x)/\mu)
    } \label{accgfsc}
  \end{equation}
  with $ (x(0),v(0)) = (x_0,v_0) \in \RR^d \times \RR^d$. 

\begin{theorem}[Strongly convex case~\cite{WRJ21,LC22}]\label{thm:cr_ac_sc}
  Suppose that $f$ is $\mu$-strongly convex. Let $(x,v) \colon [0,\infty) \to \RR^d \times \RR^d$ be the solution of~\eqref{accgfsc}.
  Then it satisfies
  \begin{equation}
		f(x(t))-f^\star \le \e{-\mur t}\paren*{f(x_0)-f^\star + \frac{\mu}{2}\normm{v_0-\xs}}.
  \end{equation}
\end{theorem}
\begin{remark}
  This system coincides with a continuous limit ODE of the accelerated gradient method for strongly convex functions~\cite{P64}
  \begin{equation}
    \ddot{x} + 2\mur \dot{x} + \nabla f(x) = 0.
  \end{equation}
\end{remark}

Discretizing these ODEs, we obtain (or reproduce) (discrete) optimization methods.
Although this has been intensively considered in such an ODE-approach in optimization research field, it is also completely natural to consult theories in numerical analysis, where rich knowledge on the discretization of ODEs are accumulated.
Among which, we here cast a spotlight on the so-called {\em discrete gradient method}, which seems a tempting candidate for optimization ODEs.

\subsubsection{Discrete gradient method for gradient systems} \label{subsubsec:DG}
The discrete gradient method~\cite{G96,MQR99} is a structure-preserving numerical integration method for general gradient systems including the gradient flow~\eqref{contgf}. 
Recall that it automatically decreases the value of objective functions (see~\eqref{pf:gengf}); there the keys were (i) the chain rule of differentiation, and (ii) the fact that the system is a gradient flow (i.e., the form of the equation itself). 
The discrete gradient (DG) method respects these keys and try to imitate them in discrete systems.

\begin{definition}[Discrete gradient~\cite{G96,QC96}]
  A continuous map $\DG f \colon \RR^d \times \RR^d \to \RR^d$ is said to be {\it discrete gradient of $f$} if the following two conditions hold for all $x,y \in \RR^d$:
  \begin{align}
    f(y) - f(x) &= \inpr{\DG f(y,x)}{y-x}, \label{dchain} \\
    \DG f(x,x) &= \nabla f(x) \label{dapprox}.
  \end{align}
\end{definition}
In the above definition, the second condition merely requests that $\DG f$ be an approximation of $\nabla f$. The first condition, on the other hand, is called {\it discrete chain rule} and is an essential condition for the key (i) above. 
The discrete chain rule is just a scalar equality constraint on the vector-valued function, and for a given $f$, there are generally infinitely many DGs. 
Some popular choices of DG will be presented in~\cref{sec:weakDG}.

Suppose we have a DG for a given $f$.
Then we can define a discrete scheme for the gradient flow~\eqref{contgf}:
\begin{equation}
  \frac{\x{k+1} - \x{k}}{h} = - \DG f \paren*{ \x{k+1},\x{k}}, \quad \x{0} = x_0,
\end{equation}
where the positive real number $ h $ is referred to as the step size,
and $\x{k}\simeq x(kh)$ is the numerical solution.
The left-hand side is an approximation of $ \dot{x} $ and is denoted by $ \dpx $ hereafter. 
Observe the definition respects the key point (ii) mentioned above.

The scheme happily decreases $f(\x{k})$ as expected:
\begin{align}
\delta^+ f \paren*{\x{k}} 
= \inpr*{\DG f \paren*{ \x{k+1},\x{k}}}{\dpx}
= -\norm*{\DG f \paren*{ \x{k+1},\x{k}}}^2
\le 0.
\end{align}
In the first equality we used the discrete chain rule, and in the second the form of the scheme itself.
Observe that the proof goes exactly in the same manner as the continuous case~\eqref{pf:gengf}.
Thus the scheme should work as an optimization method.
Note also that the above argument does not depend on the step size $h$, and it can be changed in every step (that will not destroy the decreasing property).

This might appear to achieve our goal, an integration of theories of optimization and numerical analysis.
In fact, this direction of researches has already been tried in numerical analysis community (see~\cref{subsec:prior}).
This challenge is, however, not as easy as it might seem; there remain several difficulties to counter.
First, the DG scheme generally becomes ``fully-implicit'' (in the terminology of numerical analysis), i.e., it requires nonlinear solver in each time step (as well as the existence theorem of numerical solutions; see~\cite{ERRS18}).
This is not preferable in practical optimizations, where people usually use cheap explicit methods.
Second, as pointed above, the proofs of rate estimates in the continuous systems (in~\cref{subsubsec:rates_conti}) demand the inequality of convexity~\eqref{scineq}.
Unfortunately, however, existing DGs generally do not satisfy it; see~\cref{sec:counterexample} for a counterexample.
From these reasons, the numerical analysis challenges have been so far isolated from optimization researches.

The main aim of this paper is to resolve these difficulties, and propose a truly unified framework that gives a bird-eye view on both optimization and numerical analysis research branches (and integrate them).

\subsection{Prior works}\label{subsec:prior}
\paragraph{Continuous systems for optimization.}
Early works on this aspect include the followings. The continuous gradient flow~\eqref{contgf} as a continuous optimization method was discussed in \cite{B75}. Similar arguments were then done for second-order differential equations~\cite{A00, AABR02, CEG09}. An important milestone in this direction was \cite{SBC16}, where a second-order system was shown to correspond to the famous accelerated gradient method of Nesterov~\cite{N83}. This strongly focused or revived people's attention to the direction.

For estimating convergence rates, one of the most powerful approaches nowadays is the one via Lyapunov functional~\cite{SBC16}, which gave rise to numerous successors; an incomplete partial list is~\cite{KV16, APR16, AC17, ACR18, FRV18, D19, SDSJ19, WRJ21} (see also a rich list in \cite{SRR22}). A difficulty in this approach was that the Lyapunov functionals were found only heuristically. Some remedies were considered then~\cite{SRR22, D22}, but their targets are still limited.

\paragraph{Theoretical convergence estimates.}

Theoretical analysis of the convergence rates has long been a central topic of optimization theory. 
The upper and lower bound of the convergence rates are discussed for various optimization methods and function classes in the literature. 
In this paper, we focus on the convergence rate of function values $ f \paren*{ \x{k} } - f^{\star} $, 
while those of $ \norm*{ \nabla f \paren*{ \x{k} } } $ or $ \norm*{ \x{k} - \xs } $ have also been discussed. 
Topics of particular relevance to this study include the lower bound of convergence rates of first-order methods (see \cref{rem:first-order_methods} for the relation between our framework and first-order methods) for convex and strongly convex functions: 
$ \Order{ \frac{1}{k^2} } $ for $L$-smooth and convex functions~(cf.~\cite{N18b}) 
and $ \Order{ \paren*{ 1 - \sqrt{\frac{\mu}{L}} }^{2k} } $ for $L$-smooth and $\mu$-strongly convex functions~\cite{DT22}. 
These lower bounds are tight in the sense that they are achieved by some optimization methods: 
for example, \cite{N83} for convex functions and \cite{VFL18,DT221} for strongly convex functions. 

\paragraph{Unified strategies for optimization methods.}
Unified frameworks with some theoretical guarantees include the followings.
A unified convergence analysis for an abstract optimization method class was discussed in \cite{CL21}, still it requires additional discussion for each specific method. Another systematic approach is the performance estimation problem approach, where optimization methods that are optimal in some sense are obtained by solving semidefinite programs~\cite{DT14}. A framework to analyze and design optimization methods was proposed via the concept of integral quadratic constrains~\cite{LRP16}.

\paragraph{Numerical analysis approaches for optimization.}
The link between continuous optimization methods and numerical methods  has been known from old days, but it was only quite recently that the link has been fully exploited. The possibility of discrete gradients has been investigated~\cite{GMMQS17, CEOR18, ERRS18, MSZ18, RLS18, BRS20, REQS22}. Some typical numerical methods were considered for discretizing continuous systems; Runge--Kutta methods~\cite{ZMSJ18, USM2022cheby}, symplectic methods \cite{BJW18, MJ19, SDSJ19, FSRV20, MJ21}. A numerical analysis view for Lyapunov functional approach was given~\cite{SZ21}. Limitations of continuous optimization systems in view of the stability of numerical methods were revealed~\cite{USM2022rate}.

\subsection{Contributions} \label{subsec:contri}
The main contribution is the new unified framework for the description of gradient-based optimization methods.
The key there is the proposal of {\em weak discrete gradients}, which abstracts various approximations of gradients in optimization methods.
With the concept, we propose abstract methods that are readily incorporated with convergence rate estimates.
The abstract methods include as specific instances not only typical methods in optimization such as the steepest descent and Nesterov's accelerated gradient, but also some recent methods from the numerical analysis community mentioned above. 
After this framework, it basically suffices to consider new possibilities of weak DGs to explore new gradient-based methods.

For the optimization community, the new framework gives a new view on gradient-based methods, which has the following two advantages. First, it reveals a common structure of such methods in terms of {\em weak DG}, an appropriate class of discrete approximations of gradients. Second, the conditions for being a weak DG clarify general sufficient conditions for assuring convergence rates, which have not been explicitly verbalized so far even if they have been partly and implicitly embedded in existing convergence studies.

From the numerical analysis point of view, recent numerical analysis approaches to optimization have been isolated in the optimization literature. 
The new framework gives a missing bridge; now we can catch general optimization methods via the concept of DGs, keeping the strict (original) DGs at hand.

We believe this new framework will serve as a modern working place to explore new gradient-based optimization methods, both for people in optimization and numerical analysis.
In fact, some new concrete optimization methods are immediate; for example, the application of strict DGs ((iv)--(vi) in~\cref{tab:wdg}) to the abstract methods based on the accelerated gradient flows (\cref{thm:cr_ac_conv_d,thm:cr_ac_sc_d}) gives a new class of methods, which achieves the rates in~\cref{tab:wdg}.

\section{Weak discrete gradients and the new unified framework} \label{sec:weakDG}
%
%

In this section, we define weak discrete gradients (DGs) and present our new framework.





For later use, we list some popular DGs.
When we need to distinguish them from weak DGs, we call them {\em strict DGs}.
\begin{proposition}[Strict discrete gradient]\label{thm:DG}
  The following functions are strict discrete gradients.\\
  Gonzalez discrete gradient $ \nabla_\mathrm{G} f(y,x) $~\cite{G96}:
    \begin{equation}
     \nabla f \paren*{\frac{y+x}{2}}
      + \frac{f(y) - f(x) - \inpr{\nabla f \paren*{\frac{y+x}{2}}}{y-x}}{\normm{y-x}} (y-x).
    \end{equation}
   Itoh--Abe discrete gradient $ \nabla_\mathrm{IA} f (y,x) $~\cite{IA88}:
    \begin{equation}
      \begin{bmatrix}
        \frac{f(y_1,x_2,x_3\dots,x_d) - f(x_1,x_2,x_3,\dots,x_d)}{y_1-x_1}\\
        \frac{f(y_1,y_2,x_3\dots,x_d) - f(y_1,x_2,x_3,\dots,x_d)}{y_2 - x_2}\\
        \vdots\\
        \frac{f(y_1,y_2,y_3,\dots,y_d) - f(y_1,y_2,y_3\dots,x_d)}{y_d-x_d}\\
      \end{bmatrix}.
    \end{equation}
    Average vector field (AVF) $ \nabla_\mathrm{AVF} f (y,x) $~\cite{QM08}:
    \begin{equation}
      \int_0^1 \nabla f(\tau y + (1-\tau) x) \dd \tau.
    \end{equation}
\end{proposition}

As said in~\cref{sec:intro}, they do not generally satisfy the convex inequality~\eqref{scineq}.
Therefore here we propose a concept of \emph{weak discrete gradients}, which is a relaxed version of DGs. 


\begin{definition}[Weak discrete gradient]\label{def:wdg}
  A gradient approximation\footnote{Notice that we use the notation~$\WDG$ here, distinguishing it from~$\DG$ which is reserved as the standard notation for strict discrete gradients in numerical analysis.} $\WDG f \colon \RR^d \times \RR^d \to \RR^d$ is said to be {\it weak discrete gradient of $f$} if there exists $\alpha \ge 0$ and $\beta, \gamma$ with $\beta+\gamma \ge0$ such that for all $x,y,z \in \RR^d$ the following two conditions hold:
  \begin{align}
    f(y) - f(x) &\le \inpr*{\WDG f(y,z)}{y-x} + \alpha \norm{y-z}^2 - \beta\normm{z-x} -\gamma\normm{y-x}, \label{wdgsc} \\
    \WDG f(x,x) &= \nabla f(x).
  \end{align}
\end{definition}

The condition~\eqref{wdgsc} can be understood in two ways.
First is an understanding as a discrete chain rule in a weaker sense: by substituting $x$ to $z$, we obtain an inequality
\begin{equation}
  f(y) - f(x) \le \inpr*{\WDG f(y,x)}{y-x} + (\alpha -\gamma)\normm{y-x}. \label{wchain}
\end{equation}
Compared to the strict discrete chain rule~\eqref{dchain}, it is weaker in that
it becomes an inequality and allows an error term.
Second is as a weaker discrete convex inequality:
by exchanging $x$ and $y$ and rearranging terms, we obtain another expression
\begin{align}
  f(y) - f(x) - \inpr*{\WDG f(x,z)}{y-x} \ge \beta \normm{y-x} + \gamma \normm{y-z} - \alpha\normm{x-z}. \label{wconv}
\end{align}
Compared to the strongly convex inequality~\eqref{scineq},
the term $(\mu/2)\normm{y-x}$ is now replaced  with $\beta\normm{y-x}+\gamma\normm{y-z}$, which can be interpreted as a squared distance between $y$ and the point  ($x, z$) where the gradient is evaluated. The term $ - \alpha\normm{x-z} $ is an error term.

Now we list some examples of weak DGs (the proof is left to \cref{append:wDG}). 
Notice that as concrete instances various typical gradient approximations are included.

\begin{theorem}\label{thm:wDG}
Suppose that $ f \colon \RR^d \to \RR$ is a $\mu$-strongly convex function. 
Let {\em (L)} and {\em (SC)} denote the additional assumptions: {\em (L)} $f$ is $L$-smooth, and {\em (SC)} $\mu > 0 $. 
Then, the following functions are weak discrete gradients:
\begin{enumerate}[{\em (i)}]
    \item If $\WDG f(y,x) = \nabla f(x)$ and $f$ satisfies {\em (L)}, then $(\alpha,\beta,\gamma) = (L/2,\mu/2,0)$. \label{EE}
    \item If $\WDG f(y,x) = \nabla f(y)$, then $(\alpha,\beta,\gamma) = (0,0,\mu/2)$.		\label{IE}
    \item If $\WDG f(y,x) = \nabla f(\frac{x+y}{2})$ and $f$ satisfies {\em (L)}, then $(\alpha,\beta,\gamma) = ((L+\mu)/8,\mu/4,\mu/4)$. \label{mid}
    \item If $\WDG f(y,x) = \nabla_{\mathrm{AVF}} f (y,x)$ and $f$ satisfies {\em (L)}, then $(\alpha,\beta,\gamma) = (L/6+\mu/12,\mu/4,\mu/4)$. \label{AVF}
    \item \label{Gon} If $\WDG f(y,x) = \nabla_{\mathrm{G}} f (y,x)$ and $f$ satisfies {\em (L)(SC)}, then $(\alpha,\beta,\gamma) = ((L+\mu)/8 + {(L-\mu)^2}/{16\mu}, {\mu}/{4}, 0)$.
    \item If $\WDG f(y,x) = \nabla_{\mathrm{IA}} f (y,x)$ and $f$ satisfies {\em (L)(SC)}, then $(\alpha,\beta,\gamma) = \paren{{dL^2}/{\mu}-{\mu}/{4}, {\mu}/{2}, -{\mu}/{4}}$. \label{IA}
\end{enumerate}
\end{theorem}

Note that, for the case (\ref{IE}), the same property holds even when the objective function is not differentiable by using subgradients appropriately (see the end of \cref{append:wDG}).
For the sake of brevity, however, we assume that $f$ is differentiable in this paper. 
Note also that the above statements include the case $\mu=0$, i.e., when $f$ is not strictly convex.


By using a weak DG, we define an abstract method
\begin{equation}\label{discrgf}
		\frac{\x{k+1} - \x{k} }{ h } = -\WDG f \paren*{ \x{k+1},\x{k}}, \quad \x{0} = x_0
\end{equation}
for the gradient flow~\eqref{contgf}. 
By ``abstract'' we mean that it is a formal formula, and given a concrete weak DG it reduces to a concrete method;
see~\cref{tab:wdg} which summarizes some possible choices.
Observe the abstract method covers many popular methods both from optimization and numerical analysis communities.
The step size $h$ may be selected using line search techniques, 
but for simplicity, we limit our presentation to the fixed step size in this paper (see \cref{rem:stepsize}). 

\begin{remark}\label{rem:first-order_methods}
  Notice that some weak DGs are not directly connected to the original gradient $\nabla f$'s; the Itoh--Abe weak DG (vi) even does not refer the gradient.
  Thus the concrete methods resulting from our framework do not necessarily fall into the so-called ``first-order methods,'' which run in a linear space spanned by the past gradients~\cite{N83}.
  This is why we use the terminology ``gradient-based methods'' in this paper, instead of first-order methods.
\end{remark}

Similarly to above, we can define abstract methods for ~\eqref{accgf} and \eqref{accgfsc}.
They and their theoretical results can be found in~\cref{thm:cr_ac_conv_d,thm:cr_ac_sc_d}. 

It is convenient to introduce the next lemma, which is useful to widen the scope of our framework.

\begin{lemma}\label{lem:sum}
Suppose that $f$ can be written as the sum of two functions $f_1, f_2$. 
If $\WDG_1 f_1$ and $\WDG_2 f_2$ are respectively weak discrete gradients of $f_1$ and $f_2$ with parameters $(\alpha_1,\beta_1,\gamma_1)$ and $(\alpha_2,\beta_2,\gamma_2)$, then $\WDG_1 f_1 + \WDG_2 f_2$ is a weak discrete gradient of $f$ with $(\alpha,\beta,\gamma)=(\alpha_1+\alpha_2,\beta_1+\beta_2,\gamma_1+\gamma_2)$.
\end{lemma}

This lemma allows to think of the following discretization of the gradient flow
\begin{equation}\label{pgd}
  \dpx = -\nabla f_1 \paren*{ \x{k}} - \nabla f_2 \paren*{ \x{k+1} }  
\end{equation}
within our framework.
Suppose, for exmaple, $f_1$ is $L_1$-smooth and $\mu_1$-strongly convex, and $f_2$ is $\mu_2$-strongly convex.
Then the right-hand side of~\eqref{pgd} is a weak DG with $(\alpha,\beta,\gamma) = (L_1/2,\mu_1/2,\mu_2/2)$.
In this case, the scheme is known as the proximal gradient method or the forward-backward splitting algorithm in optimization (cf.~\cite{BC17b}).
By discretizing the accelerated gradient flows, we can obtain accelerated versions. (Acceleration of the proximal gradient method has been long considered~\cite{BT09,BT09_1}.)


\begin{table*}[htbp]
\begin{center}
\caption{Examples of weak discrete gradients and resulting convergence rates when $f$ is a $\mu$-strongly convex and $L$-smooth function on $\RR^d$. The numbers in the $\WDG f$ column correspond to the numbers in \cref{thm:wDG}. The line of the proximal gradient method is described in the setting of~\eqref{pgd}. (DG) represents that this weak discrete gradient is a strict discrete gradient. The convergence rates written in the table are for the best choice of step sizes in \cref{thm:cr_sc_d,thm:cr_ac_sc_d}.}
\scalebox{0.95}{
\renewcommand{\arraystretch}{1.4}
\begin{tabular}{c|c|c|cc} \hline
          & Optimization method            &                  & \multicolumn{2}{c} {Convergence rates} \\
$\WDG f$ & (in the case of~\eqref{contgf}) & Numerical method & \cref{thm:cr_sc_d} & \cref{thm:cr_ac_sc_d} \\ \hline
\eqref{EE} & steepest descent & explicit Euler & $\Order{ \paren*{ 1 - 2 \frac{\mu}{L + \mu}}^k}$ & $ \Order{ \paren*{ 1 - \sqrt{\frac{\mu}{L}} }^k} $ \\
\eqref{IE} & proximal point method & implicit Euler & $0$ & $0$ \\
\eqref{EE}+\eqref{IE} & proximal gradient method & (splitting method) & $ \Order{ \paren*{ 1 - 2 \frac{\mu_1 + \mu_2 }{L_1+\mu_1+2\mu_2} }^k} $ & $ \Order{ \paren*{ 1 - \sqrt{ \frac{ \mu_1 + \mu_2 }{L_1 + \mu_2} }}^k} $ \\
\eqref{mid} & --- & implicit midpoint & $ \Order{ \paren*{ 1 - 8 \frac{\mu}{L+7\mu} }^k}$ & $ \Order{ \paren*{ 1 - \sqrt{\frac{4\mu}{L+3\mu}}}^k } $ \\
\eqref{AVF} & --- & AVF (DG) & $ \Order{ \paren*{ 1 - 6 \frac{\mu}{L+5\mu} }^k}$ & $ \Order{ \paren*{ 1 - \sqrt{\frac{3\mu}{L+2\mu}} }^k} $ \\
\eqref{Gon} & --- & Gonzalez (DG) & $ \Order{ \paren*{ 1 - 8 \frac{\mu^2}{L^2+4\mu^2} }^k} $ & $ \Order{ \paren*{ 1 - \sqrt{\frac{4\mu^2}{L^2+3\mu^2}} }^k} $ \\
\eqref{IA} & --- & Itoh--Abe (DG) & $ \Order{ \paren*{ 1 - 2 \frac{ \mu^2 }{4d^2L^2 - \mu^2} }^k} $ & $ \Order{ \paren*{ 1 - \sqrt{\frac{\mu^2}{ 4dL^2-2\mu^2 }} }^k} $ \\ \hline
\end{tabular}
\renewcommand{\arraystretch}{1}}
\label{tab:wdg}
\end{center}
\end{table*}

\section{Convergence rates of the abstract optimization methods}
\label{sec:rates_discrete}

In this section, we establish the discrete counterparts of \cref{thm:cr_conv,thm:cr_sc,thm:cr_ac_conv,thm:cr_ac_sc}. 
See \cref{app:proof_discrete} for the proofs. 

Part of the results in~\cref{subsec:gf,subsec:agf},
more precisely the best convergence rates in strongly convex cases, are summarized in \cref{tab:wdg} for readers' convenience.
For convex cases, the order of the convergence rate does not depend on the choice of weak discrete gradients (DGs); 
the coefficients depend on the inverse of the upper bound of the step size. 
In other words, the smaller the value of $\alpha$ is, the better the method is. 

\begin{remark}\label{rem:stepsize}
In the following theorems, we assume that the step size $h$ is fixed for brevity. 
However, they can be easily extended to the case with variable step sizes. 
In that case, in order to keep the Lyapunov argument alive, it is sufficient to assure that basically all the step sizes respect the step size condition (finite number of violations would not affect the theory). 
In addition, it must be separately guaranteed that there is a positive lower bound of the step sizes. 
\end{remark}


\subsection{For the abstract method based on the gradient flow} \label{subsec:gf}

The abstract method was given in~\eqref{discrgf}.

\begin{theorem}[Convex case]\label{thm:cr_conv_d}
Let $\WDG f$ be a weak discrete gradient of $f$ and suppose that $\beta \ge 0, \gamma \ge 0$. 
Let $f$ be a convex function which additionally satisfies the necessary conditions that the weak DG requires.
Let $ \setE*{ \x{k} } $ be the sequence given by \eqref{discrgf}. 
Then, under the step size condition $h \le 1/(2\alpha)$, the sequence satisfies
\begin{equation}
f\paren*{\x{k}} - f^\star \le \frac{\normm{x_0 - x^\star}}{2kh}.
\end{equation}
\end{theorem}
\begin{remark}
Convexity of $f$ is translated to the assumption $\beta \ge 0, \gamma \ge 0$. 
The proof is done assuming $\beta=\gamma=0$, which corresponds to $\mu=0$ in the continuous setting.
\end{remark}

\begin{remark}
Since \cref{thm:cr_conv_d} gives the convergence rate of the abstract method~\eqref{discrgf}, 
it contains various individual convergence theorems. 
This is illustrated below using the proximal gradient method~\eqref{pgd} as an example. 
Suppose that $f_1$ is $L_1$-smooth and convex, and $f_2$ is convex. 
Then, $ \WDG f (y,x) = \nabla f_1 (x) + \nabla f_2 (y) $ is a weak discrete gradient with the parameter $ ( \alpha, \beta , \gamma) = (L_1/2,0,0) $ due to \cref{thm:wDG,lem:sum}. 
Therefore, the proximal gradient method~\eqref{pgd} satisfies the assumption of \cref{thm:cr_conv_d} and thus the convergence rate is $ \Order{ \frac{1}{k} } $ under the step size condition $ h \le \frac{1}{L_1} $. 
\end{remark}

\begin{theorem}[Strongly convex case]\label{thm:cr_sc_d}
	Let $\WDG f$ be a weak discrete gradient of $f$ and suppose that $\beta + \gamma >0$.
  Let $f$ be a strongly convex function which additionally satisfies the necessary conditions that the weak DG requires.
 Let $ \setE*{ \x{k} } $ be the sequence given by \eqref{discrgf}. 
	Then, under the step size condition $h \le 1/(\alpha+\beta)$, the sequence satisfies
	\begin{equation}
		f\paren*{\x{k}} - f^\star \le \paren*{1-\frac{2(\beta+\gamma) h}{1+2\gamma h}}^k \normm{x_0 - x^\star}.
	\end{equation}
        In particular, the sequence satisfies 
        \begin{equation}
		f\paren*{\x{k}} - f^\star \le \paren*{1-\frac{2(\beta+\gamma)}{\alpha+\beta+2\gamma}}^k \normm{x_0 - x^\star},
	\end{equation}
        when the optimal step size $h=1/(\alpha+\beta)$ is employed. 
\end{theorem}

\subsection{For the abstract methods based on the accelerated gradient flows}\label{subsec:agf}
We consider abstract methods with weak DGs based on the accelerated gradient flows~\eqref{accgf} and~\eqref{accgfsc}, which will be embedded in the theorems below.
Here we note one thing: when using \eqref{wdgsc} as an approximation of the chain rule, we can determine $z$ independently of $x$ and $y$, which gives us some degrees of freedom.
In the following theorems we show some choices of $\z{k}$, which were found in the proof so that it can be easily calculated from known values and the Lyapunov function decreases monotonically.

\begin{theorem}[Convex case]\label{thm:cr_ac_conv_d}
Let $\WDG f$ be a weak discrete gradient of $f$ and suppose that $\beta \ge 0,\gamma \ge 0$. 
Let $f$ be a convex function which additionally satisfies the necessary conditions that the weak DG requires.
 Let $ \setE*{ \paren*{ \x{k}, \vv{k} } } $ be the sequence given by 
	\begin{equation}
		\simulparen{
			\begin{aligned}
				\dpx &= \frac{\delp A_{k}}{A_k}\paren*{\vv{k+1}-\x{k+1}}, \\
				\dpv &= -\frac{\delp A_k}{4} \WDG f\paren*{\x{k+1},\z{k}},\\
        \frac{\z{k}-\x{k}}{h} &= \frac{\delp A_k}{A_{k+1}}\paren*{\vv{k}-\x{k}}
			\end{aligned}
		}
	\end{equation}
  with $\paren*{ \x{0}, \vv{0}} = (x_0,v_0)$,
  where $A_k := A(kh)$.
  Then if $A_k = (kh)^2$ and $h\le 1/\sqrt{2\alpha}$, the sequence satisfies
	\begin{equation}
		f\paren*{\x{k}} - f^\star \le \frac{2\norm*{x_0 - x^\star}^2}{A_k}.
	\end{equation}
\end{theorem}

\begin{theorem}[Strongly convex case]\label{thm:cr_ac_sc_d}
    Let $\WDG f$ be a weak discrete gradient of $f$ and suppose that $\beta+\gamma>0$. 
     Let $f$ be a strongly convex function which additionally satisfies the necessary conditions that the weak DG requires.
    Let $ \setE*{ \paren*{ \x{k}, \vv{k} } } $ be the sequence given by 
	\begin{equation}
		\simulparen{
				\dpx &= \sqrt{2(\beta+\gamma)}\paren*{\vksxks},\\
				\dpv &= \sqrt{2(\beta+\gamma)} \left( \frac{\beta}{\beta+\gamma}\z{k} + \frac{\gamma}{\beta+\gamma}\x{k+1} -\vv{k+1}-\frac{\wdgz}{2(\beta+\gamma)} \right), \\
        \frac{\z{k} - \x{k}}{h} &= \sqrt{2(\beta+\gamma)}\paren*{ \x{k} + \vv{k} - 2\z{k}}
		}
	\end{equation}
  with $\paren*{\x{0}, \vv{0}}= (x_0,v_0)$.
  Then if 
  \begin{equation}
    h \le \overline{h} \coloneqq \frac{1}{\sqrt{2}(\sqrt{\alpha+\gamma} - \sqrt{\beta+\gamma})},
  \end{equation}
  the sequence satisfies
	\begin{align}
		f \paren*{ \x{k} } - f^\star 
        \le \paren*{ 1 + \sqrt{2(\beta+\gamma)}h}^{-k}\paren*{f(x_0) - f^\star + \beta \normm{v_0 - \xs}}.
	\end{align}
  In particular, the sequence satisfies
  \begin{align}
		f \paren*{ \x{k} } - f^\star 
        \le \paren*{1 - \sqrt{\frac{\beta+\gamma}{\alpha+\gamma}}}^k \paren*{f(x_0) - f^\star + \beta \normm{v_0 - \xs}}, 
  \end{align}
  when the optimal step size $ h = \overline{h}$ is employed. 
\end{theorem}
\begin{remark}
  Time scaling would eliminate the factor $\sqrt{2(\beta+\gamma)}$ from the scheme and simplify it~\cite{LC22}, but we do not do so to match the time scale with the accelerated gradient method and to make the correspondence with the continuous system visible.
\end{remark}

\section{Discussions}
\label{sec:discussions}

\paragraph{Tightness of the theories.}
Since our target in this paper is a unified framework that focuses on common structures, it is not surprising if it happens that some theories are not tight in each concrete method included.
Some comments on this aspect are in order.

Let us consider~\cref{thm:cr_conv_d} with the choice (i) $\WDG f (y,x)=\nabla f(x)$, i.e., the simplest steepest descent case for convex objective functions.
Since $\alpha = L/2$ (\cref{thm:cr_conv}), the step size condition is $h \le 1/(2\alpha) = 1/L.$
This gives a stronger sufficient condition than the well-known condition $h \le 2/L.$
The gap should be recovered by more fully utilizing the explicit form of the steepest descent, which is outside the scope of the framework; still the authors hope to point out that the condition $1/L$ correctly reflects the effect of $L$.

On strongly convex functions, the proposed scheme in \cref{thm:cr_ac_sc_d} with the choice (i) achieves the convergence rate $ \Order{ \paren*{ 1 - \sqrt{\frac{\mu}{L}} }^{k} } $. 
Though the method is a first-order method, this rate does not coincide with the optimal rate $ \Order{ \paren*{ 1 - \sqrt{\frac{\mu}{L}} }^{2k} } $ (see \cref{subsec:prior}). 
However, the rate achieved by the proposed method is nearly optimal and indeed coincides with the rate of Nesterov's accelerated gradient method for strongly convex functions.


\paragraph{Polyak--{\L}ojasiewicz type functions.} 
\label{par:PL}
The weak discrete gradients can be useful also for non-convex functions. Here we illustrate it by taking functions satisfying the Polylak--{\L}ojasiewicz (P{\L}) condition.
A function $f$ is said to satisfy the P{\L} condition if 
\begin{equation}
    - \normm{\nabla f(x)} \le - 2 \mu \paren*{ f(x) - f^{\star} }
\end{equation}
holds for any $ x \in \RR^d$. 
This was introduced as a sufficient condition for the steepest descent to converge~\cite{P63}. 
The set of functions satisfying the P{\L} condition contains all differentiable strongly convex functions and some nonconvex functions such as $ f(x) = x^2 + 3 \sin^2 (x) $. 

\begin{theorem}[Continuous systems]\label{thm:cr_PL}
  Suppose that $f$ satisfies the P{\L} condition. Let $x \colon [0,\infty) \to \RR^d$ be the solution of the gradient flow~\eqref{contgf}. 
  Then the solution satisfies
  \begin{equation}	
		f(x(t)) - f^\star \le \e{-\mu t}\normm{x_0-x^\star}.
  \end{equation}
\end{theorem}

Let us define another weak discrete gradient for functions satisfying the P{\L} condition. 
Recall that the first condition of weak discrete gradients (\cref{def:wdg}) has two meanings: the discrete chain rule~\eqref{wchain} and the discrete convex inequality~\eqref{wconv}.
One could consider the P{\L} condition instead of convexity.

\begin{definition}\label{def:PLwDG}
  A gradient approximation $\WDG f \colon \RR^d \times \RR^d \to \RR^d$ is said to be {\it P{\L}-type weak discrete gradient of $f$} if there exists positive numbers $\alpha, \beta$ such that for all $x,y \in \RR^d$ the following three conditions hold:
  \begin{align}
    \WDG f(x,x) &= \nabla f(x), \\
    f(y) - f(x) &\le \inpr{\WDG f(y,x)}{y-x} + \alpha \norm{y-x}^2, \label{wPLchain}\\
    - \norm*{\WDG f(y,x)} &\le - \sqrt{2\mu (f(x) - f^\star)} + \beta \norm{y-x}. \label{wPL}
  \end{align}
\end{definition}

\begin{theorem}\label{thm:PLwDG}
  If $f$ is $L$-smooth and satisfies the P{\L} condition with the parameter $\mu$, the following functions are P{\L}-type weak discrete gradients:
  \begin{enumerate}[{\em (i)}]
    \item $\WDG f(y,x) = \nabla f(x)$; then $(\alpha,\beta) = (L/2,0)$. \label{PLEE}
		\item If $\WDG f(y,x) = \nabla f(y)$, then $(\alpha,\beta) = (L/2,L)$.		\label{PLIE}
		\item If $\WDG f(y,x) = \nabla f(\frac{x+y}{2})$, then $(\alpha,\beta) = (L/8,L/2)$. \label{PLmid}
		\item \label{PLAVF} If $\WDG f(y,x) = \nabla_{\mathrm{AVF}} f (y,x)$, then $(\alpha,\beta) = (0,L/2)$.
		\item If $\WDG f(y,x) = \nabla_{\mathrm{G}} f (y,x)$, then $(\alpha,\beta) = (0,5L/8)$. \label{PLGon}
		\item If $\WDG f(y,x) = \nabla_{\mathrm{IA}} f (y,x)$, then $(\alpha,\beta) = (0,\sqrt{d}L)$. \label{PLIA}
  \end{enumerate}
\end{theorem}
\begin{remark}
  The parameters $\alpha$ and $\beta$ imply the magnitude of the discretization error. 
  As the second condition in \cref{def:PLwDG}, we can adopt instead
  \begin{equation}
    - \norm*{\WDG f(y,x)} \le - \sqrt{2\mu(f(y) - f^\star)} + \beta \norm{y-x}. \label{wPL2}
  \end{equation}
  Then, we obtain better parameters for the implicit Euler method~\eqref{PLIE}.
\end{remark}
The proof of \cref{thm:PLwDG} is postponed in \cref{append:PLwDG}.

\begin{theorem}[Discrete systems]\label{thm:cr_PL_d}
  Let $\WDG f$ be a P{\L}-type weak discrete gradient of $f$.
  Let $f$ be a function which satisfies the necessary conditions that the P{\L}-type weak DG requires.
  Let $ \setE*{ \x{k} } $ be the sequence given by \eqref{discrgf}. 
  Then, under the step size condition $h \le 1/\alpha$, the sequence satisfies
  \begin{equation}
		f \paren*{\x{k}} - f^\star \le \paren*{1- 2\mu h \frac{ ( 1 -\alpha h)}{(1 + \beta h)^2}}^k \paren*{ f \paren*{x_0} - f^\star}.
  \end{equation}
  In particular, the sequence satisfies
  \begin{equation}
		f \paren*{\x{k}} - f^\star \le \paren*{1 - \frac{\mu}{2(\alpha + \beta)}}^k \paren*{ f \paren*{x_0} - f^\star},
  \end{equation}
  when the optimal step size $h = 1/(2 \alpha + \beta)$ is employed. 
\end{theorem}

\paragraph{Optimization methods that cannot be described by the new framework.}
The proposed framework is flexible enough to describe wide range of gradient-based optimization methods, but it fails to catch some.
For example, it seems some splitting algorithms do not fit the framework such as the Douglas--Rachford splitting method~\cite{EB92}:
\begin{align}
  \frac{\x{k+1/2} - \x{k}}{h} &= -\WDG_1 f_1 \paren*{ \x{k+1/2},\x{k}},\\
  \frac{\x{k+1} - \x{k+1/2}}{h} &= -\WDG_2 f_2 \paren*{ \x{k+1},\x{k+1/2}}.
\end{align}
It seems difficult to write the right-hand side by a single weak DG.

It should be also noted that some existing methods from numerical analysis community do not fit the framework as is.
For example, it seems methods based on Runge--Kutta numerical methods are not easy to express by weak DGs.

Extending the framework for adapting to the above cases is an interesting research direction.

\section{Conclusions and perspectives}
\label{sec:conclusion}

In this paper we proposed a new unified framework for describing gradient-based optimization methods.
It is based on weak discrete gradients (DGs), an extension of discrete gradients which are a powerful tool in structure-preserving numerical methods.
The extension was done so that weak DGs can describe various popular methods in optimization.
As a result, the framework can describe wide range of gradient-based optimization methods investigated in the optimization literature, and also in a branch of numerical analysis where numerical analysis approach to optimization was considered; the framework gives a bird-eye view on them.
The authors hope the framework serves as a good working stage to further explore gradient-based optimization methods.

Some new methods are immediate from the framework both as methods and for their convergence estimates.
For example, strict DG methods applied to the accelerated gradient flows are new, and its theory is given in this paper for the first time~(see~\cref{subsec:contri}).

Due to the restriction of space, no numerical experiments were included in this paper.
The authors have carried out some preliminary numerical tests and confirmed the theory (see \cref{sec:num_exp}).
It was also found out that some new methods can be competitive with state-of-the-art methods, such as Nesterov's accelerated gradient.
Our next goal is to find promising instances within the framework through such trials.


\bibliography{main}
\bibliographystyle{apalike}

\newpage
\appendix
\onecolumn

\section{Proofs of theorems in \cref{subsubsec:rates_conti}}
\label{app:proof_conti}

\subsection{Proof of \cref{thm:cr_conv}}
\label{subsec:proof_cr_conv}

It is sufficient to show that a Lyapunov function
\begin{equation}
    E(t) \coloneqq t(f(x(t)) - f^\star) + \frac{1}{2}\normm{x(t)-x^\star}
\end{equation}
is monotonically nonincreasing since
\begin{equation}
    f(x(t)) - f^\star \le \frac{E(t)}{t} \le \frac{E(0)}{t} = \frac{\normm{x(0)-x^\star}}{2t}.
\end{equation}
Indeed,
\begin{align}
    \dot{E} &=t\inpr{\nabla f(x)}{\dot{x}} + f(x) - f^\star + \inpr{x-x^\star}{\dot{x}}\\
    &= -t\normm{\nabla f(x)} + f(x) - f^\star - \inpr{\nabla f(x)}{x-x^\star}\\
    &\le 0
\end{align}
holds. 
Here, at each line, we applied the Leibniz rule and the chain rule, substituted the ODE, and used the convexity of $f$ in this order. 
\qed

\subsection{Proof of \cref{thm:cr_sc}}

It is sufficient to show a Lyapunov function
\begin{equation}
    E(t) \coloneqq \e{\mu t}\paren*{f(x(t))-f^\star + \frac{\mu}{2}\normm{x(t)-\xs}}
\end{equation}
is monotonically nonincreasing and thus it is sufficient to show that $\tilde{E}(t) := \e{-\mu t} E(t)$ satisfies	$\dot{\tilde{E}} \le -\mu \tilde{E}$.
Indeed, 
\begin{align}
    \dot{\tilde{E}} &= \inpr{\nabla f(x)}{\dot{x}} + \mu \inpr{x-\xs}{\dot{x}}\\
    &= - \normm{\nabla f(x)} - \mu\inpr{\nabla f(x)}{x-\xs} \\
    &\le - \normm{\nabla f(x)} - \mu\paren*{f(x) - f^\star + \frac{\mu}{2}\normm{x-\xs}}\\
    &\le -\mu\tilde{E}
\end{align}
holds. 
Here, at each line, we applied the chain rule, substituted the ODE, and used strong convexity of $f$ in this order. 
\qed

\subsection{Proof of \cref{thm:cr_ac_conv}}

  It is sufficient to show that
  \begin{equation}
    E(t) \coloneqq A(t)(f(x(t)) - f^\star) + 2\norm{v(t)-x^\star}^2
  \end{equation}
  is nonincreasing. Actually,
  \begin{align}
    \dot{E} &= \dot{A}(f(x)-f^\star) + A \dt{ (f(x) - f^\star) }+ \dt{(2\norm{v(t)-x^\star}^2)}\\
    &= \dot{A}(f(x)-f^\star) + A \inpr{\nabla f(x)}{\dot{x}} + 4\inpr{\dot{v}}{v-x^\star}\\
    &= \dot{A}(f(x)-f^\star) + A \inpr*{\nabla f(x)}{\frac{\dot{A}}{A}(v - x)} - 4\inpr*{\frac{\dot{A}}{4}\nabla f(x)}{v-x^\star} \\
    &=\dot{A}(f(x) - f^\star - \inpr{\nabla f(x)}{x - x^\star})\\
    &\le 0
  \end{align}
  holds.
  Here, at each line, we applied the Leibniz rule, used the chain rule, substituted the ODE, and used the convexity of $f$ in this order.
  \qed

\subsection{Proof of \cref{thm:cr_ac_sc}}

  It is sufficient to show that 
  \begin{equation}
		{E}(t) \coloneqq \e{\mur t}\paren*{f(x) - f^\star + \frac{\mu}{2}\normm{v-\xs}}
	\end{equation}
  is nonincreasing and thus it is sufficient to show that $\tilde{E}(t):=\e{-\mur t} E(t)$ satisfies $\dot{\tilde{E}} \le -\mur \tilde{E}$.
  Actually,
	\begin{align}
		\dot{\tilde{E}} &= \inpr{\nabla f(x)}{\dot{x}} + \mu\inpr{\dot{v}}{v-\xs}\\
		&= \inpr{\nabla f(x)}{\mur (v-x)} + \mu\inpr{\mur(x-v-\nabla f(x)/\mu)}{v-\xs}\\
		&= \mur(\inpr{\nabla f(x)}{\xs-x} - \mu\inpr{v-x}{v-\xs})\\
		&= \mur\paren*{\inpr{\nabla f(x)}{\xs-x} - \frac{\mu}{2}(\normm{v-x}+\normm{v-\xs}-\normm{x-\xs})}\\
		&\le -\mur \paren*{\paren*{f(x) - f^\star + \frac{\mu}{2}\normm{v-\xs}} - \frac{\mu}{2}\normm{v-x}}\\
		&\le -\mur \tilde{E}
	\end{align}
holds.
Here, at each line, we applied the chain rule, substituted the ODE, rearranged terms, decomposed the inner product by the law of cosines (see~\cref{app:cosines}), and used the strong convexity of $f$.
\qed

\section{Strict discrete gradients and convexity}
\label{sec:counterexample}

In this section, we describe that strict discrete gradients are not generally compatible with the convex inequality; this complements the discussion in~\cref{subsubsec:DG}.
For example, let us consider imitating the discussion in \cref{subsec:proof_cr_conv} by using a discrete gradient scheme $ \delta^+ \x{k} = - \nabla_{\mathrm{d}} f \paren*{ \x{k+1}, \x{k} } $. 
Then, since we use the inequality
\begin{equation}
f(x) - f^{\star} - \inpr*{ \nabla f(x) }{ x - x^{\star} } \le 0
\end{equation}
that holds due to the convexity of $f$, 
we should ensure that the discrete counterpart of the left-hand side
\begin{equation}
    f(x) - f^{\star} - \inpr*{ \nabla_{\mathrm{d}} f (y,x) }{ x - \xs }
\end{equation}
is nonpositive for any $ x, y \in \RR^d$. 
However, there is a simple counterexample as shown below. 

Let us consider a quadratic and convex objective function $ f(x) = \frac{1}{2} \inpr*{ x }{Qx} $, where $Q \in \RR^{d \times d}$ is a positive definite matrix. 
In this case, $ \xs = 0 $ and $ f^{\star} = 0 $ hold. 
Then, when we choose a discrete gradient $ \nabla_{\mathrm{d}} f (y,x) = \nabla_{\mathrm{G}} f (y,x) = \nabla_{\mathrm{AVF}} f (y,x) = Q \left(\frac{y+x}{2}\right) $, 
we see 
\begin{equation}
    f(x) - f^{\star} - \inpr*{ \nabla_{\mathrm{d}} f (y,x) }{ x - \xs }
    = \frac{1}{2} \inpr*{x}{Qx} - \inpr*{ Q \left(\frac{y+x}{2}\right)  }{ x }
    = - \frac{1}{2} \inpr*{y}{Qx},
\end{equation}
which is positive when $ y = - x $ and $ x \neq 0 $.

\section{Proof of \cref{thm:wDG}}\label{append:wDG}

\begin{enumerate}[(i)]
\item
Since we assume $f$ is $L$-smooth and $ \mu$-strongly convex, 
\begin{align}
  f(y) - f(z) &\le \inpr{\nabla f (z)}{y-z} + \frac{L}{2} \normm{y-z},\\
  f(z) - f(x) &\le \inpr{\nabla f (z)}{z-x} - \frac{\mu}{2} \normm{z-x}
\end{align}
holds for any $x,y,z \in \RR^d$. 
By adding each side of these inequalities, we obtain
\begin{equation}
  f(y) - f(x) \le \inpr{\nabla f(z)}{y-x} + \frac{L}{2} \normm{y-z} - \frac{\mu}{2} \normm{z-x}. \label{3ptsd}
\end{equation}
(This inequality is known as the three points descent lemma in optimization.)

\item It follows immediately from the $\mu$-strong convexity of $f$.

\item 
By replacing $z$ in \eqref{3ptsd} with $\theta y+(1-\theta)z$, and invoking \cref{lem:para}, we have
\begin{align}
  &f(y) - f(x) - \inpr*{\nabla f\paren*{\theta y+(1-\theta)z}}{y-x} \\
  &\quad \le \frac{L}{2}\normm{y-(\theta y + (1-\theta)z)} - \frac{\mu}{2} \normm{\theta y + (1-\theta)z - x} \\
  &\quad = \frac{L}{2}(1-\theta)^2 \normm{y-z} - \frac{\mu}{2}  (-\theta(1-\theta)\normm{y-z} + (1-\theta)\normm{z-x} + \theta\normm{y-x}).
\end{align}
Especially when $\theta = 1/2$, $(\alpha,\beta,\gamma) = (L/8+\mu/8,\mu/4,\mu/4)$.

\item
By the same calculation as (\ref{mid}), we obtain
\begin{align}
  &f(y) - f(x) - \inpr*{\int_0^1 \nabla f(\tau y + (1-\tau)z) \dd \tau}{y-x}\\
  &\quad = \int_0^1 \left[ f(y) - f(x) - \inpr*{\nabla f(\tau y + (1-\tau)z)}{y-x} \right] \dd \tau\\
  &\quad \le \int_0^1 \left[ \frac{L}{2}(1-\tau)^2 \normm{y-z} - \frac{\mu}{2} (-\tau(1-\tau)\normm{y-z} + (1-\tau)\normm{z-x} + \tau\normm{y-x}) \right] \dd \tau\\
  &\quad = \paren*{\frac{L}{6} + \frac{\mu}{12}}\normm{y-z} - \frac{\mu}{4} \normm{z-x} - \frac{\mu}{4}\normm{y-x}.
\end{align}

\item
By \eqref{3ptsd}, we obtain
\begin{align}
  f(y) - f(z) - \inpr*{\nabla f\paren*{\frac{y+z}{2}}}{y-z}
  \le \frac{L}{2}\normmb{y-\frac{y+z}{2}} - \frac{\mu}{2} \normmb{\frac{y+z}{2}-z}
  =\frac{L-\mu}{8}\normm{y-z}.
\end{align}
Since this inequality holds with $y$ and $z$ swapped, we have
\begin{equation}
  \abs*{f(y) - f(z) - \inpr*{\nabla f\paren*{\frac{y+z}{2}}}{y-z}}
  \le \frac{L-\mu}{8}\normm{y-z}.
\end{equation}

Thus,
\begin{align}
  &f(y) - f(x) - \inpr{\nabla_{\mathrm{G}}f(y,z)}{y-x} \\
  &\quad = f(y) - f(x) - \inpr*{\nabla f\paren*{\frac{y+z}{2}} + \frac{f(y) - f(z) - \inpr*{\nabla f\paren*{\frac{y+z}{2}}}{y-z}}{\normm{y-z}}(y-z)}{y-x} \\
  &\quad \le f(y) - f(x) - \inpr*{\nabla f\paren*{\frac{y+z}{2}}}{y-x} + \abs*{\frac{f(y) - f(z) - \inpr*{\nabla f\paren*{\frac{y+z}{2}}}{y-z}}{\normm{y-z}}}\abs{\inpr{y-z}{y-x}}\\
  &\quad \le \frac{L}{2}\normmb{y-\frac{y+z}{2}} - \frac{\mu}{2} \normmb{\frac{y+z}{2} - x} + \frac{L-\mu}{8} \paren*{\frac{L-\mu}{8\mu}\normm{y-z} + \frac{2\mu}{L-\mu}\normm{y-x}} \\
  &\quad = \frac{L}{8}\normm{y-z} - \frac{\mu}{2}\paren*{\frac{1}{2}\normm{y-x} + \frac{1}{2}\normm{z-x} - \frac{1}{4}\normm{y-z}} + \frac{(L-\mu)^2}{16\mu}\normm{y-z} + \frac{\mu}{4}\normm{y-x}\\
  &\quad = \paren*{\frac{L}{8}+\frac{\mu}{8}+\frac{(L-\mu)^2}{16\mu}}\normm{y-z} - \frac{\mu}{4}\normm{z-x}
\end{align}
holds,
where the second inequality follows from the arithmetic-geometric means (AM-GM) inequality.

\item
In the following, for $y,z \in \RR^d$ and $k= 2,\dots,d$, $y_{1:k-1}z_{k:d}$ denotes a vector $(y_1,\dots,y_{k-1},z_{k},\dots,z_d)^\top \in \RR^d$, while $y_{1:0}z_{1:d}$ and $y_{1:d}z_{d+1:d}$ denote $z$ and $y$, respectively. 
By the telescoping sum and $\mu$-strong convexity of $f$, we obtain
\begin{align}
  &f(y) - f(x) - \inpr{\nabla_{\mathrm{IA}} f(y,z)}{y-x} \\
  &\quad = f(y) - f(x) - \sum_{k=1}^{d} \frac{f(y_{1:k}z_{k+1:d}) - f(y_{1:k-1}z_{k:d})}{y_k-z_k}(y_k-x_k) \\
  &\quad = f(y) - f(x) - \sum_{k=1}^{d} \frac{f(y_{1:k}z_{k+1:d}) - f(y_{1:k-1}z_{k:d})}{y_k-z_k}(y_k-z_k+z_k-x_k) \\
  &\quad = f(z) - f(x) - \sum_{k=1}^{d} \frac{f(y_{1:k}z_{k+1:d}) - f(y_{1:k-1}z_{k:d})}{y_k-z_k}(z_k-x_k) \\
  &\quad \le \inpr{\nabla f(z)}{z-x} - \frac{\mu}{2}\normm{z-x} \\
  &\qquad\quad + \sum_{k=1}^{d}
  \begin{cases}
    \displaystyle
    \frac{(\nabla f(y_{1:k-1}z_{k:d}))_k(z_k-y_k) - \frac{\mu}{2}(z_k-y_k)^2}{y_k-z_k} (z_k-x_k) & \text{if}\quad \displaystyle \frac{z_k - x_k}{y_k-z_k} > 0 \\
    \displaystyle
    \frac{(\nabla f(y_{1:k}z_{k+1:d}))_k(y_k-z_k) - \frac{\mu}{2}(y_k-z_k)^2}{z_k-y_k} (z_k-x_k) & \text{if}\quad \displaystyle \frac{z_k - x_k}{y_k-z_k} < 0
  \end{cases}\\
  &\quad = \inpr{\nabla f(z)}{z-x} - \frac{\mu}{2}\normm{z-x} \\
  &\qquad\quad - \sum_{k=1}^{d}
  \begin{cases}
    (\nabla f(y_{1:k-1}z_{k:d}))_k (z_k-x_k) + \frac{\mu}{2}(y_k-z_k) (z_k-x_k) & \text{if}\quad \displaystyle \frac{z_k - x_k}{y_k-z_k} > 0\\
    (\nabla f(y_{1:k}z_{k+1:d}))_k(z_k-x_k) + \frac{\mu}{2}(z_k-y_k)(z_k-x_k) & \text{if}\quad \displaystyle \frac{z_k - x_k}{y_k-z_k} < 0
  \end{cases}\\
  &\quad = \sum_{k=1}^{d} 
  \begin{cases}
    ((\nabla f(z))_k-(\nabla f(y_{1:k-1}z_{k:d}))_k)(z_k-x_k) & \text{if}\quad \displaystyle \frac{z_k - x_k}{y_k-z_k} > 0 \\
    ((\nabla f(z))_k-(\nabla f(y_{1:k}z_{k+1:d}))_k)(z_k-x_k) & \text{if}\quad \displaystyle \frac{z_k - x_k}{y_k-z_k} < 0
  \end{cases}\\
  &\qquad\quad - \frac{\mu}{2}\normm{z-x} - \frac{\mu}{2}\sum_{k=1}^{d} \abs{z_k-y_k}\abs{z_k-x_k}.
\end{align}
To evaluate the first term of the most right-hand side, we use the following inequalities:
\begin{align}
  \abs{(\nabla f(z))_k-(\nabla f(y_{1:k-1}z_{k:d}))_k} &\le L\norm{z-y_{1:k-1}z_{k:d}} \le L\norm{z-y},\\
  \abs{(\nabla f(z))_k-(\nabla f(y_{1:k}z_{k+1:d}))_k} &\le L\norm{z-y_{1:k}z_{k+1:d}} \le L\norm{z-y},
\end{align}
which hold due to the $L$-smoothness of $f$,
and also
\begin{equation}
  \sum_{k=1}^d \abs{z_k-x_k} \le \sqrt{d \sum_{k=1}^d \abs{z_k-x_k}^2} = \sqrt{d} \norm{z-x},
\end{equation}
which holds due to Jensen's inequality. 
Using them, we obtain
\begin{align}
    f(y) - f(x) - \inpr{\WDG_{\mathrm{IA}} f(y,z)}{y-x} 
    &\le \sqrt{d} L \norm{z-y} \norm{z-x} - \frac{\mu}{2}\normm{z-x} - \frac{\mu}{2}\sum_{k=1}^{d} \abs{z_k-y_k}\abs{z_k-x_k} \\
    &\le \frac{dL^2}{\mu}\normm{z-y} + \frac{\mu}{4}\normm{z-x} - \frac{\mu}{2}\normm{z-x} - \frac{\mu}{2}\sum_{k=1}^{d} \abs{z_k-y_k}\abs{z_k-x_k} ,
\end{align}
where the last inequality holds due to the AM-GM inequality. 
Finally, the last term is bounded by
\begin{align}
  \frac{\mu}{2}\sum_{k=1}^d \abs{z_k-y_k}\abs{z_k-x_k}
  &= -\frac{\mu}{4}\sum_{k=1}^d \paren*{\abs{z_k-y_k}^2 + \abs{z_k-x_k}^2 - \abs{\abs{z_k-y_k} - \abs{z_k-x_k}}^2}\\
  &\le -\frac{\mu}{4}\sum_{k=1}^d \paren*{\abs{z_k-y_k}^2 + \abs{z_k-x_k}^2 - \abs{y_k-x_k}^2}\\
  &= -\frac{\mu}{4} \paren*{\normm{z-y} + \normm{z-x} - \normm{y-x}}.
\end{align}
This proves the theorem. 
\end{enumerate}
\qed

For \eqref{IE}, as noted in the above proof, the assumption of differentiability of $f$ is unnecessary, let alone $L$-smoothness. Since $f$ is a proper convex function on $\RR^d$ in our setting, the subdifferential $\rd f(x)$ is nonempty for all $x\in\RR^d$. Thus we can use $\WDG f(y,x) \in \rd f(y)$ instead of $\WDG f(y,x) = \nabla f(y)$. By definition of subgradients, we can recover the same parameters $(\alpha,\beta,\gamma)$ as the differentiable case.

If $\mu = 0$, the proofs for \eqref{Gon} and \eqref{IA} cease to work where we apply the AM-GM inequality to the inner product. 
This is also pointed out in the main body of the paper.

\section{Proofs of theorems in \cref{sec:rates_discrete}}
\label{app:proof_discrete}

\subsection{Proof of \cref{thm:cr_conv_d}} 

It is sufficient to show that the discrete Lyapunov function
\begin{equation}
    E^{(k)} \coloneqq kh \paren*{ f \paren*{ \x{k}}-f^\star} + \frac{1}{2}\norm*{\x{k} - \xs}^2
\end{equation}
is nonincreasing. Actually,
\begin{align}
    &\dpE \\
    &\quad = kh \paren*{ \delp f \paren*{\x{k}} } + f \paren*{ \x{k+1} } - f^\star + \delp \paren*{\frac{1}{2} \norm*{\x{k}-\xs}^2}\\
    &\quad \le kh \paren*{ \inpr*{\WDG f \paren*{ \x{k+1},\x{k}}}{\dpx} + \alpha h\norm*{\dpx}^2} + f \paren*{ \x{k+1}} - f^\star + \inpr*{\x{k+1}-x^\star}{\dpx} - \frac{h}{2} \norm*{\dpx}^2\\
    &\quad = -kh(1-\alpha h)\norm*{\WDG f \paren*{ \x{k+1},\x{k}}}^2 + f \paren*{ \x{k+1}} - f^\star - \inpr*{\WDG f \paren*{ \x{k+1},\x{k}}}{\x{k+1}-x^\star}- \frac{h}{2}\normm{\dpx}\\
    &\quad \le -kh(1-\alpha h)\normm{\WDG f \paren*{\x{k+1},\x{k}}} - \paren*{\frac{h}{2}-\alpha h^2}\normm{\dpx}
\end{align}
holds, and thus if $h \le 1/(2\alpha)$, the right-hand side is not positive. 
Here, at each line, we applied the discrete Leibniz rule, the weak discrete gradient condition~\eqref{wdgsc}, \cref{lem:cos} as the chain rule, substituted the scheme, and applied again \eqref{wdgsc} as the convex inequality.
 \qed

\subsection{Proof of \cref{thm:cr_sc_d}}

Let 
\begin{equation}
    \tE{k} \coloneqq f\paren*{\x{k}} - f^\star + (\beta+\gamma)\normm{\xkxs}.
\end{equation}
If $\dptE \le -c \tE{k+1}$ for $c>0$, it can be concluded that $\E{k} = (1+ch)^k \tE{k}$ is nonincreasing, and hence $f(\x{k}) - f^\star \le (1+ch)^{-k}\E{0}$. Actually,
\begin{align}
    &\dptE \\ 
    &\quad = \delp f \paren*{ \x{k}} + \delp \paren*{(\beta+\gamma)\normm{\x{k}-\xs}}\\
    &\quad \le \inpr*{\WDG f \paren*{ \x{k+1},\x{k}}}{\dpx} + (\alpha-\gamma) h\normm{\dpx}+ 2(\beta+\gamma)\inpr*{\x{k+1}-x^\star}{\dpx} - (\beta+\gamma) h\normm{\dpx}\\
    &\quad = -(1-(\alpha-\gamma) h + (\beta+\gamma)h)\normm{\WDG f \paren*{ \x{k+1},\x{k}}} - 2(\beta+\gamma) \inpr*{\x{k+1}-x^\star}{\WDG f \paren*{ \x{k+1},\x{k}}} \\
    &\quad \le -2(\beta+\gamma) \paren*{f \paren*{\x{k}} - f^\star + \beta\normm{\x{k}-\xs} + \gamma\normm{\x{k+1}-\xs}} \\
    &\qquad -\paren*{ 1-(\alpha-\gamma) h + (\beta+\gamma)h - 2\alpha(\beta+\gamma)h^2}\normm{\WDG f\paren*{\x{k+1},\x{k}}} \label{gdscpr}
\end{align}
 holds.
  Here, after the second line we used the weak discrete gradient condition~\eqref{wdgsc} as the chain rule, substituted the scheme and used \eqref{wdgsc} as the strongly convex inequality.
  
  Now we aim to bound $\normm{\xkxs}$ with $\normm{\xksxs}$. By the same calculation for $\delp \normm{\xkxs}$ as above, we get the evaluation
	\begin{equation}
		\delp \normm{\xkxs} 
			\le -2(f(\x{k+1})-f^\star +\beta\normm{\xkxs} + \gamma\normm{\xksxk}) 
			 - (h-2\alpha h^2)\normm{\WDG f\paren*{\x{k+1},\x{k}}}.
		\label{gdscnorm}
	\end{equation}
	Thus, if $h \le 1/(2\alpha)$, we get $\normm{\xksxs} \le \normm{\xkxs}$. In this case, since the second term of \eqref{gdscpr} is nonpositive, it follows that 
	\begin{equation}
		\dptE \le -2(\beta + \gamma)\tE{k+1}.
	\end{equation}
	
  To obtain a better rate which is included in the statement of the theorem, by directly using \eqref{gdscnorm} for \eqref{gdscpr}, we see	\begin{align}
		\dptE &\le -\frac{2(\beta+\gamma)}{1-2\beta h}\paren*{f(\x{k+1})-f^\star +(\beta+\gamma)\normm{\xksxs}}\\
		&\quad -\paren*{\frac{1-2\alpha h}{1-2\beta h} 2(\beta+\gamma)\beta h^2 + 1-(\alpha-\gamma) h + (\beta+\gamma)h - 2\alpha(\beta+\gamma)h^2} \normm{\WDG f\paren*{\x{k+1},\x{k}}} .
	\end{align}
	Since the second term of the right-hand side is nonpositive under $h\le 1/(\alpha+\beta)$, it can be concluded that 
	\begin{equation}
		\dptE \le -\frac{2(\beta+\gamma)}{1-2\beta h}\tE{k+1}.
	\end{equation}
	In this case the convergence rate is 
	\begin{equation}
		\paren*{\frac{1}{1+\frac{2(\beta+\gamma)h}{1-2\beta h}}}^k = \paren*{1 - \frac{2(\beta+\gamma)h}{1+2\gamma h}}^k.
	\end{equation}
\qed

\subsection{Proof of \cref{thm:cr_ac_conv_d}}

  It is sufficient to show that
\begin{equation}
    E^{(k)} \coloneqq A_k \paren*{ f \paren*{ \x{k} }-f^\star} + 2\norm{\vv{k}-\xs}^2
\end{equation}
  is nonincreasing. Actually,
  \begin{align}
    &\dptE \\
    &\quad = (\delp A_{k}) \paren*{ f \paren*{ \x{k+1}} -f^\star} + A_k \paren*{ \delp f \paren*{ \x{k} }} + 2\delp \paren*{ \norm*{\vv{k}-\xs}^2}\\
    &\quad \le (\delp A_{k})\paren*{ f \paren*{ \x{k+1}} -f^\star} +A_k\inpr*{\WDG f \paren*{\x{k+1},\z{k}}}{\dpx}\\
    &\qquad +4\inpr*{\dpv}{\vv{k+1}-\xs} + \frac{A_k}{h}\alpha \norm*{\x{k+1}-\z{k}}^2 - 2h\norm*{\dpv}^2 \\ 
    &\quad \le (\delp A_{k})\paren*{ f \paren*{ \x{k+1}} -f^\star} +A_k\inpr*{\WDG f\paren*{\x{k+1},\z{k}}}{\frac{\delp A_{k}}{A_k}\paren*{\vv{k+1}-\x{k+1}}}\\
    &\qquad - 4\inpr*{\frac{\delp A_k}{4} \WDG f \paren*{ \x{k+1},\z{k}}}{\vv{k+1}-\xs} + \frac{A_k}{h}\alpha \norm*{\x{k+1}-\z{k}}^2 - 2h\norm*{\dpv}^2 \\ 
    &\quad = (\delp A_{k}) \paren*{ f \paren*{ \x{k+1}}-f^\star -\inpr*{\WDG f\paren*{ \x{k+1},\z{k}}}{\x{k+1}-\xs}} + \frac{A_k}{h}\alpha \norm*{\x{k+1}-\z{k}}^2 - 2h\norm*{\dpv}^2 \\  	
    &\quad \le (\delp A_{k})\alpha\normm{\x{k+1}-\z{k}} + \frac{A_{k}}{h}{\alpha} \norm*{\x{k+1}-\z{k}}^2 - 2h\norm*{\dpv}^2 \\
    &\quad = \frac{1}{h} \paren*{ A_{k+1}\alpha \norm*{\x{k+1}-\z{k}}^2 - 2\norm*{\vv{k+1}-\vv{k}}^2 } =: (\mathrm{err}). \label{accpr}
  \end{align}
  Here, at each line, we used the discrete Leibniz rule, applied \eqref{wdgsc} and \cref{lem:cos} as the chain rule, substituted the scheme, and applied again \eqref{wdgsc} as the convex inequality.

  Now we define $\z{k}$ so that $(\mathrm{err}) \le 0$ holds.
  When $\z{k} := \x{k+1}$, $(\mathrm{err})$ becomes nonpositive without step size constraints. 

	Or, since using the scheme we can write
	\begin{align}
		\normm{\x{k+1}-\z{k}} = \normmb{\frac{A_{k+1}-A_k}{A_{k+1}}\vv{k+1} + \frac{A_k}{A_{k+1}}\x{k} - \z{k}},
	\end{align}
  by setting
	\begin{equation}
    \z{k} := \frac{A_{k+1}-A_k}{A_{k+1}}\vv{k} + \frac{A_k}{A_{k+1}}\x{k},
  \end{equation}
we obtain
\begin{equation}
	h \times (\mathrm{err}) = \paren*{\frac{(A_{k+1}-A_k)^2}{A_{k+1}}\alpha - 2}\normm{\vksvk}.
\end{equation}
This choice of $\z{k}$ is shown in the theorem.
When, for example, $A_k = (kh)^2$, $(\mathrm{err}) \le 0$, provided that $h\le 1/\sqrt{2\alpha}$.
Here we see that only up to a quadratic function is allowed as $A_k$ if $\alpha > 0$.
 \qed

\subsection{Proof of \cref{thm:cr_ac_sc_d}}

  It is sufficient to show that 
  \begin{equation}
    \tilde{E}^{(k)} \coloneqq f \paren*{ \x{k} } - f^\star + (\beta+\gamma) \normm{\vv{k} - \xs}
  \end{equation}
  satisfies $\dptE \le -\sqrt{m}\tE{k+1}$.
  To simplify notation, $2(\beta + \gamma)$ is written as $m$, and the error terms are gathered into $(\mathrm{err})$.
  Then, we see
\begin{align}
  \dptE
  & =\delp f \paren*{\x{k}} + \frac{m}{2} \delp \normm{\vv{k} - \xs}\\
  & \le \inpr*{\wdgz}{\dpx} + \frac{\alpha}{h} \normm{\xksz} - \frac{\beta}{h} \normm{\zxk} - \gamma h \normm{\dpx}\\
  &\quad + m \inpr*{\dpv}{\vksxs} - \frac{m}{2}h \normm{\dpv}\\
  & = \inpr*{\wdgz}{\sqrt{m} \paren*{ \vksxks }} + (\mathrm{err})\\
  &\quad + m \inpr*{\sqrt{m}\paren*{\frac{2\beta}{m}\z{k} + \frac{2\gamma}{m}\x{k+1}-\vv{k+1}-\frac{1}{m} \wdgz}}{\vksxs} \\
  & = \sqrt{m}\inpr*{\wdgz}{\xs-\x{k+1}} - 2\sqrt{m}\beta \inpr*{\vksz}{\vksxs} \\
  &\quad - 2\sqrt{m}\gamma\inpr*{\vksxks}{\vksxs} + (\mathrm{err})\\
  &= \sqrt{m}\inpr*{\wdgz}{\xs-\x{k+1}} - \sqrt{m}\beta \paren*{\normm{\vksz} + \normm{\vksxs} - \normm{\z{k}-\xs}}\\
  &\quad - \sqrt{m}\gamma \paren*{\normm{\vksxks} + \normm{\vksxs} - \normm{\xksxs}} + (\mathrm{err})\\
  &\le -\sqrt{m}\paren*{f\paren*{\x{k+1}} - f^\star + \frac{m}{2}\normm{\vksxs}}\\
  &\quad + \sqrt{m} \paren*{ \alpha \normm{\xksz} - \beta \normm{\vksz} - \gamma\normm{\vksxks}} + (\mathrm{err})\\
  &= -\sqrt{m}\tilde{E}^{(k+1)} + (\mathrm{err}).
\end{align}
Here, the first inequality follows from \eqref{wdgsc} as the chain rule, the second equality from the substitution of the form of the method, and the second inequality follows from again \eqref{wdgsc} as the strongly convex inequality.
In the second inequality, we also used
\begin{equation}
-\inpr*{\wdgz}{\x{k+1}-\xs} + \beta \normm{\z{k}-\xs} + \gamma \normm{\xksxs} \le - \paren*{ f \paren*{ \x{k+1}}-f^\star} + \alpha\normm{\xksz}.
\end{equation}

Now we define $\x{k}$ so that $(\mathrm{err}) \le 0$.
An obvious choice is $\z{k} := \x{k+1}$, where $(\mathrm{err})$ is nonpositive under any step size.

To derive another definition of $\z{k}$, we proceed with the calculation of the error terms by substituting the form of the method:
\begin{align}
  h \times (\mathrm{err})
	&= \alpha \normm{\xksz} - \beta \normm{\zxk} - \gamma \normm{\xksxk} - (\beta+\gamma)\normm{\vksvk}\\
	&\quad + \sqrt{m}h \paren*{ \alpha  \normm{\xksz} - \beta \normm{\vksz} - \gamma \normm{\vksxks}} \\
  &= \alpha (h+1)\normm{\xksz} - \beta \paren*{ \normm{\zxk} + \normm{\vksvk} + \sqrt{m}h\normm{\vksz} }\\
	&\quad - \gamma \paren*{ \normm{\xksxk} + \normm{\vksvk} + \sqrt{m}h\normm{\vksxks}}\\
  &=  \alpha (h+1)\normm{\xksz} \\
	&\quad - \beta \paren*{\normm{\zxk} + \normm{\vksvk} + \sqrt{m}h\normmb{\x{k+1} +\frac{\xksxk}{\sqrt{m}h} - \z{k}}} \\
	&\quad - \gamma \left(\normm{\xksxk} + \normm{\vksvk} + \sqrt{m}h\normmb{\x{k+1} +\frac{\xksxk}{\sqrt{m}h} - \x{k+1}} \right).
\end{align}
Hereafter, $\sqrt{m}h$ is denoted by $\tilde{h}$.
By using \cref{lem:para}, we have
\begin{align}
  &\tilde{h}\normmb{\x{k+1} + \frac{\xksxk}{\tilde{h}} - \z{k}}\\
  &\quad = \tilde{h}\paren*{\frac{\tilde{h}+1}{\tilde{h}}}^2\normmb{\frac{\tilde{h}}{\tilde{h}+1}\paren*{\xksz} + \frac{1}{\tilde{h}+1}\paren*{\xksxk}}\\
  &\quad = \frac{\paren*{\tilde{h}+1}^2}{\tilde{h}} \paren*{\frac{\tilde{h}}{\tilde{h}+1}\normm{\xksz} + \frac{1}{\tilde{h}+1}\normm{\xksxk} - \frac{\tilde{h}}{\paren*{\tilde{h}+1}^2}\normm{\zxk}}\\
  &\quad = \paren*{\tilde{h}+1}\normm{\xksz} + \frac{\tilde{h}+1}{\tilde{h}} \normm{\xksxk} - \normm{\zxk}.
\end{align}
Thus, we see
\begin{align}
  \tilde{h}\times(\mathrm{err}) =  (\alpha - \beta) \paren*{ \tilde{h}+1} \normm{\xksz} - (\beta + \gamma) \paren*{\frac{\tilde{h}+1}{\tilde{h}}\normm{\xksxk} + \normm{\vksvk}}.
\end{align}
Here when we define $\z{k}:=\x{k}$, $(\mathrm{err})$ is nonpositive under the condition
\begin{equation}
  (\alpha-\beta)\paren*{ \tilde{h}+1} - (\beta+\gamma)\frac{\tilde{h}+1}{\tilde{h}} \le 0.
\end{equation}
This condition reads as $\tilde{h} \le (\beta+\gamma)/(\alpha-\beta)$, and the convergence rate is
\begin{equation}
  \paren*{ 1+\sqrt{2(\beta+\gamma)}}^{-k} = \paren*{ 1+\tilde{h}}^{-k} \ge \paren*{1-\frac{\beta+\gamma}{\alpha+\gamma}}^k.
\end{equation}

To obtain a better rate, we continue the computation of $(\mathrm{err})$ without defining $\z{k}$. Let
\begin{equation}
  \eta = \frac{1}{\frac{\tilde{h}+1}{\tilde{h}}+1}=\frac{\tilde{h}}{2\tilde{h}+1},
\end{equation}
and again by inserting the form of the method, and by using \cref{lem:para},
\begin{align}
  &\frac{\tilde{h}+1}{\tilde{h}}\normm{\xksxk} + \normm{\vksvk} \\
  &\quad = \frac{\tilde{h}+1}{\tilde{h}}\normm{\xksxk}+\normmb{\x{k+1} + \frac{\xksxk}{\tilde{h}} - \vv{k}} \\
  &\quad = \frac{\tilde{h}+1}{\tilde{h}}\normm{\xksxk} + \paren*{\frac{\tilde{h}+1}{\tilde{h}}}^2\normmb{\x{k+1}- \frac{\tilde{h}}{\tilde{h}+1}\vv{k}-\frac{1}{\tilde{h}+1}\x{k}}\\
  &\quad = \frac{\tilde{h}+1}{\tilde{h}}\frac{1}{\eta} \paren*{\eta \normm{\xksxk} + \ia \normmb{\x{k+1}- \frac{\tilde{h}}{\tilde{h}+1}\vv{k}-\frac{1}{\tilde{h}+1}\x{k}}}\\
  &\quad = \frac{\tilde{h}+1}{\tilde{h}}\frac{1}{\eta} \left( \normmb{\eta\paren*{\xksxk} + \ia \paren*{\x{k+1}- \frac{\tilde{h}}{\tilde{h}+1}\vv{k}-\frac{1}{\tilde{h}+1}\x{k}}} \right. \\
  &\qquad\qquad\qquad \left. + \eta\ia\normmb{\xksxk - \paren*{\x{k+1}- \frac{\tilde{h}}{\tilde{h}+1}\vv{k}-\frac{1}{\tilde{h}+1}\x{k}}} \right) \\
  &\quad = \frac{\tilde{h}+1}{\tilde{h}}\frac{2\tilde{h}+1}{\tilde{h}} \left[ \normmb{\x{k+1} - \paren*{\frac{\tilde{h}+1}{2\tilde{h}+1}\x{k} + \frac{\tilde{h}}{2\tilde{h}+1}\vv{k}}} + \frac{\tilde{h}}{2\tilde{h}+1}\frac{\tilde{h}+1}{2\tilde{h}+1}\paren*{\frac{\tilde{h}}{\tilde{h}+1}}^2 \normm{\vv{k}-\x{k}} \right].
\end{align}
Hence we obtain 
\begin{equation}
	\begin{split}
		\tilde{h} \times (\mathrm{err})
		& =(\alpha-\beta)\paren*{\tilde{h}+1}\normm{\xksz} - (\beta + \gamma)\frac{\tilde{h}}{2\tilde{h}+1} \normm{\vv{k}-\x{k}}\\
		&\quad - (\beta + \gamma)\frac{\tilde{h}+1}{\tilde{h}}\frac{2\tilde{h}+1}{\tilde{h}} \normmb{\x{k+1} - \frac{\tilde{h}+1}{2\tilde{h}+1}\x{k} - \frac{\tilde{h}}{2\tilde{h}+1}\vv{k}}.
	\end{split}
\end{equation}
If we set
\begin{equation}
	\z{k} := \frac{\tilde{h}+1}{2\tilde{h}+1}\x{k} + \frac{\tilde{h}}{2\tilde{h}+1}\vv{k},
\end{equation}
$(\mathrm{err})$ is nonpositive under the condition
\begin{equation}
  (\alpha-\beta)\paren*{\tilde{h}+1} - (\beta+\gamma)\frac{\tilde{h}+1}{\tilde{h}}\frac{2\tilde{h}+1}{\tilde{h}} \le 0.
\end{equation}
The definition of $\z{k}$ is shown in the theorem.
By solving the above inequality, we obtain the step size limitation
\begin{equation}
  \tilde{h} \le \frac{\sqrt{\beta + \gamma}}{\sqrt{\alpha + \gamma} - \sqrt{\beta + \gamma}},
\end{equation}
which is shown in the theorem.
\qed

\section{Law of cosines and parallelogram identity} \label{app:cosines}
This section summarizes some useful lemmas used in the preceding sections.

In Hilbert spaces, especially in Euclidean spaces, the law of cosines holds:
\begin{equation}
  \normm{y-x} = \normm{y} + \normm{x} - 2\inpr{y}{x}.
\end{equation}
In this paper, we use this formula as an error-containing discrete chain rule of the squared norm.
\begin{lemma}\label{lem:cos}
  For all $\x{k+1},\x{k} \in \RR^d$,
  \begin{equation}
    \normm{\x{k+1}} - \normm{\x{k}} = 2\inpr{\x{k+1}}{\x{k+1}-\x{k}} - \normm{\x{k+1}-\x{k}}
  \end{equation}
\end{lemma}

Another famous identity for the Hilbert norm (especially the Euclidean norm) is the parallelogram identity:
\begin{equation}
  \normmb{\frac{x+y}{2}} + \normmb{\frac{x-y}{2}} = \frac{1}{2}(\normm{x} + \normm{y}).
\end{equation}
In this paper, we use a generalization of this identity.
In the following lemma, we recover the parallelogram identity by setting $\alpha = 1/2$.
\begin{lemma}\label{lem:para}
  For all $x,y \in \RR^d$ and $\alpha \in \RR$,
  \begin{equation}
    \normm{\alpha x + (1-\alpha) y} = \alpha \normm{x} + (1-\alpha) \normm{y} - \alpha (1-\alpha) \normm{x-y}.
  \end{equation}
\end{lemma}
\begin{proof}
  The claim is obtained by adding each side of the following equalities: 
  \begin{align}
    \normm{\alpha x + (1-\alpha) y} &= \alpha^2 \normm{x} + (1-\alpha)^2 \normm{y} + 2 \alpha(1-\alpha) \inpr{x}{y},\\
    \alpha (1-\alpha) \normm{x-y} &= \alpha (1-\alpha) \normm{x} + \alpha (1-\alpha) \normm{y} - 2 \alpha (1-\alpha) \inpr{x}{y}.
  \end{align}
\end{proof}

\section{Proofs of theorems in \cref{sec:discussions}}

\subsection{Proof of \cref{thm:cr_PL}}

It is sufficient to show that 
\begin{equation}
    E(t) \coloneqq f(x(t)) - f^\star
\end{equation}
satisfies $ \dot{E} \le - 2 \mu E $. Indeed,
\begin{align}
    \dot{E} &= \inpr{\nabla f(x)}{\dot{x}} 
    = - \normm{\nabla f(x)} 
    \le -2\mu(f(x) - f^\star) 
    = -2\mu E
\end{align}
holds. 
Here, we used the chain rule, the continuous system itself, the P{\L} condition, and the definition of $E$ in this order.
\qed

\subsection{Proof of \cref{thm:PLwDG}}\label{append:PLwDG}
By the P{\L} condition, we observe that
\begin{align}
  -\norm*{\WDG f(y,x)} &\le -\sqrt{2\mu (f(x)-f^\star)} + \norm*{\nabla f(x)} - \norm*{\WDG f(y,x)}\\
  & \le -\sqrt{2\mu (f(x)-f^\star)} + \norm*{\nabla f(x) - \WDG f(y,x)}.
\end{align}
Thus, 
the evaluation of $\norm*{\WDG f(y,x) - \nabla f(x)}$ yields $\beta$.

\begin{enumerate}[(i)]
\item
From $L$-smoothness $\alpha = L/2$ follows. By the definition $\beta = 0$.

\item 
$L$-smoothness yields $\alpha = L/2$ and $\beta = L$. (Note that the convexity of $f$ would imply $\alpha=0$, but it is not assumed now. If we adopt the other definition~\eqref{wPL2} then $\beta = 0$.)

\item
By the same application of $L$-smoothness, we obtain $\alpha = L/8$ and $\beta = L/2$.

\item
Since the discrete chain rule exactly holds, $ \nabla_{\mathrm{AVF}} f $ satisfies $ \alpha = 0 $. 
Then, by the $L$-smoothness of $f$, 
\begin{align}
  \norm*{\nabla_{\mathrm{AVF}} f(y,x) - \nabla f(x)}
  &= \norm*{\int_0^1 \nabla f(\tau y + (1-\tau)x) \dd \tau - \nabla f(x)}\\
  &\le \int_0^1 \norm*{\nabla f(\tau y + (1-\tau)x) - \nabla f(x)} \dd \tau\\
  &\le \int_0^1 L \norm{\tau y + (1-\tau) x - x} \dd \tau\\
  &\le \int_0^1 L \tau \norm{y-x} \dd \tau\\
  &\le \frac{L}{2} \norm{y-x}
\end{align}
holds, which implies $\beta = L/2$. 

\item 
Similar to the case (\ref{PLAVF}), $ \alpha = 0 $ holds. 
By the $L$-smoothness of $f$, 
\begin{align}
  \norm*{\nabla_{\mathrm{G}} f(y,x) - \nabla f(x)}
  & = \norm*{\nabla f\paren*{\frac{y+x}{2}} - \frac{f(y) - f(x) - \inpr*{\nabla f\paren*{\frac{y+x}{2}}}{y-x}}{\normm{y-x}}(y-x) - \nabla f(x)}\\
  &\le \norm*{\nabla f\paren*{\frac{y+x}{2}} - \nabla f(x)} + \frac{\abs{f(y) - f(x) - \inpr*{\nabla f\paren*{\frac{y+x}{2}}}{y-x}}}{\norm{y-x}} \\
  &\le \frac{L}{2}\norm{y-x} + \frac{L}{8}\norm{y-x}\\
  &= \frac{5L}{8}\norm{y-x},
\end{align}
which implies $ \beta = 5L/8 $. 

\item 
Similar to the previous cases, $ \alpha = 0 $ holds. 
Using the same notation as in \cref{append:wDG} (\ref{IA}), we obtain
\begin{align}
  \norm*{\nabla_{\mathrm{IA}} f(y,x) - \nabla f(x)}
  &= \norm*{
    \begin{bmatrix}
      \frac{f(y_1,x_2,x_3\dots,x_d) - f(x_1,x_2,x_3,\dots,x_d)}{y_1-x_1}\\
      \frac{f(y_1,y_2,x_3\dots,x_d) - f(y_1,x_2,x_3,\dots,x_d)}{y_2-x_2}\\
      \vdots\\
      \frac{f(y_1,y_2,y_3,\dots,y_d) - f(y_1,y_2,y_3\dots,x_d)}{y_d-x_d}\\
    \end{bmatrix}
    - \nabla f(x)} \\
  &= \norm*{
    \begin{bmatrix}
      \partial_1 f(\theta_1 y_1 + (1-\theta_1) x_1,x_2,x_3\dots,x_d)\\
      \partial_2 f(y_1,\theta_2 y_2 + (1-\theta_2) x_2,x_3\dots,x_d)\\
      \vdots\\
      \partial_d f(y_1,y_2,y_3,\dots,\theta_d y_d + (1-\theta_d) x_d) \\
    \end{bmatrix}
    - \nabla f(x)} \\
  & = \sqrt{ \sum_{k=1}^d \abs*{ (\nabla f(\theta_k y_{1:k}x_{k+1:d} + (1-\theta_k) y_{1:k-1}x_{k:d}))_k - (\nabla f(x))_k }^2} \\
  &\le \sqrt{\sum_{k=1}^d L^2\normm{y-x}} \\
  &= \sqrt{d}L \norm{y-x},
\end{align}
where $\theta_k \in [0,1]$ is a constant by the mean value theorem. 
Therefore, $\beta = \sqrt{d} L $ holds. 
\end{enumerate}
\qed

\subsection{Proof of \cref{thm:cr_PL_d}}

   Let 
  \begin{equation}
    \tE{k} \coloneqq f \paren*{\x{k}} - f^\star.
  \end{equation}
  If $\dptE \le -c \tE{k}$ for $c>0$, it can be concluded that $\E{k} = (1-ch)^{-k} \tE{k}$ is nonincreasing and hence $f \paren*{\x{k}} - f^\star \le (1-ch)^{k}\E{0}$.
  Before starting the computation of $\dptE$, we transform the weak discrete P{\L} condition~\eqref{wPL} into a more convenient form. By substituting the scheme into \eqref{wPL}, we obtain
  \begin{equation}
    -\normm{\WDG f(\x{k+1},\x{k})} \le -\frac{\gamma}{(1 + \beta h)^2} \paren*{ f \paren*{\x{k}} - f^\star}.
  \end{equation}
  Thus, it follows from the weak discrete chain rule~\eqref{wPLchain}, the scheme, and the above inequality, that
  \begin{align}
    \dptE &= \delp f \paren*{ \x{k}}\\
    &\le \inpr*{\WDG f \paren*{ \x{k+1},\x{k}}}{ \delta^+ \x{k} } + \alpha h \normm{\delta^+ \x{k}}\\
    &= -(1 - \alpha h)\normm{\WDG f\paren*{ \x{k+1},\x{k}}} \\
    &\le - \paren*{ 1 - \alpha h} \frac{\gamma}{(1+ \beta h)^2} \paren*{ f\paren*{ \x{k}} - f^\star}\\
    &= -\gamma\frac{1-\alpha h}{(1 + \beta h)^2} \E{k}.
  \end{align}
  Hence if $h \le 1/\alpha$ we have the convergence.
  \qed

\section{Some numerical examples}\label{sec:num_exp}
In this section, we give some numerical examples to complement the discussion in the main body of this paper.
Note that this is just to illustrate that we can actually easily construct new concrete methods just by assuring the conditions of weak discrete gradients (weak DGs), and that the resulting methods in fact achieve the prescribed rates; we here do not intend to explore a method that beats known state-of-the-art methods. It is of course an ultimate goal of the unified framework project, but is left as an important future work.

Below we consider some explicit optimization methods derived as special cases of the abstract weak DG methods.
Here we pick up simple two-dimensional problems so that we can observe not only the decrease of the objective functions but also the trajectories of the points $x$'s for our intuitive understandings.

First, we consider the case where the objective function is a $L$-smooth convex function.
An explicit weak discrete gradient method is then found as
\begin{equation}
  \simulparen{
    \begin{aligned}
      \x{k+1} - \x{k} &= \frac{2k+1}{k^2} \paren*{ \vv{k+1}-\x{k+1}}, \\
      \vv{k+1} - \vv{k} &= -\frac{2k+1}{4} h^2 \nabla f \paren*{ \z{k} },\\
      \z{k}-\x{k} &= \frac{2k+1}{(k+1)^2} \paren*{ \vv{k}-\x{k}}.
    \end{aligned}
  }\label{5wDGc}
\end{equation}
We call this method (wDG-c). 
This is the simplest example of the abstract method in~\cref{thm:cr_ac_conv_d}, where we choose $A_k = (kh)^2$ and $\WDG f(y,x) = \nabla f (x)$.
The authors believe this method itself has not been explicitly pointed out in the literature, and is new.
The expected rate is the one predicted in the theorem, $\Order{1/k^2}$, under the step size condition $h \le 1/\sqrt{L}$ (recall that $\alpha$ for the weak DG is $L/2$ as shown in~\cref{thm:wDG}).
For comparison, we pick up Nesterov's accelerated gradient method for convex functions
\begin{equation}
  \simulparen{
    \y{k+1} &= \x{k} - h^2 \nabla f \paren*{ \x{k} },\\
    \x{k+1} &= \y{k+1} + \frac{k}{k+3} \paren*{ \y{k+1}-\y{k}}.
  }\label{5NAGc}
\end{equation}
We denote this method by (NAG-c). 
It is well-known that it achieves the same rate, under the same step size condition; we summarize these information in~\cref{table:5ex}. 

As an objective function, we employ
\begin{equation}
  f(x) = 0.1{x_1}^4 + 0.001{x_2}^4, \label{obj5c}
\end{equation}
which is not strongly convex.
(Strictly speaking, this is not $L$-smooth as well, but we consider in the following way: for each initial $x$ we obtain the level set $\{ x\, | \, f(x) = f(\x{0})\}$. Then we consider the function~\eqref{obj5c} inside the region, and extend the function outside it appropriately; for example such that the function grows linearly as $\|x\|\to\infty$.)

Numerical results are shown in \cref{fig:c}.
The top-left panel of the figure shows the convergence of the objective function $f(x)$ when the optimal step size $1/\sqrt{L}$ is chosen.
We see both methods achieve the predicted rate $\Order{1/k^2}$ (mind the dotted guide line).
We also see that under this setting, (wDG-c) converges faster than (NAG-c).
This suggests that the new framework can give rise to an optimization method that is competitive to state-of-the-art methods (as said before, we do not say anything conclusive on this point; in order to discuss practical performance, we further need to discuss other implementation issues such as stepping schemes.)
The trajectories of $\x{k}$'s are almost the same, for all the tested step sizes.

\begin{table}[htbp]
  \caption{Step size limitations and convergence rates of the methods used in the experiments}
  \begin{center}
    \begin{tabular}{c|c|c} \hline
      scheme & step size limitation & convergence rate \\
      \hline
      NAG-c~\eqref{5NAGc} & $1/\sqrt{L}$ & $\Order{\frac{1}{k^2}}$ \\
      wDG-c~\eqref{5wDGc} & $1/\sqrt{L}$ & $\Order{\frac{1}{k^2}}$ \\
      \hline
      NAG-sc~\eqref{5NAGsc} & $1/\sqrt{L}$ & $\Order{\paren*{1-\sqrt{\frac{\mu}{L}}}^k}$ \\
      wDG-sc~\eqref{5wDGsc} & $1/(\sqrt{L}- \sqrt{\mu})$ & $\Order{\paren*{1-\sqrt{\frac{\mu}{L}}}^k}$ \\
      wDG2-sc~\eqref{5wDGscnoy} & $\sqrt{\mu}/(L-\mu)$ & $\Order{\paren*{1-\frac{\mu}{L}}^k}$ \\
      \hline
    \end{tabular}
  \end{center}
  \label{table:5ex}
\end{table}

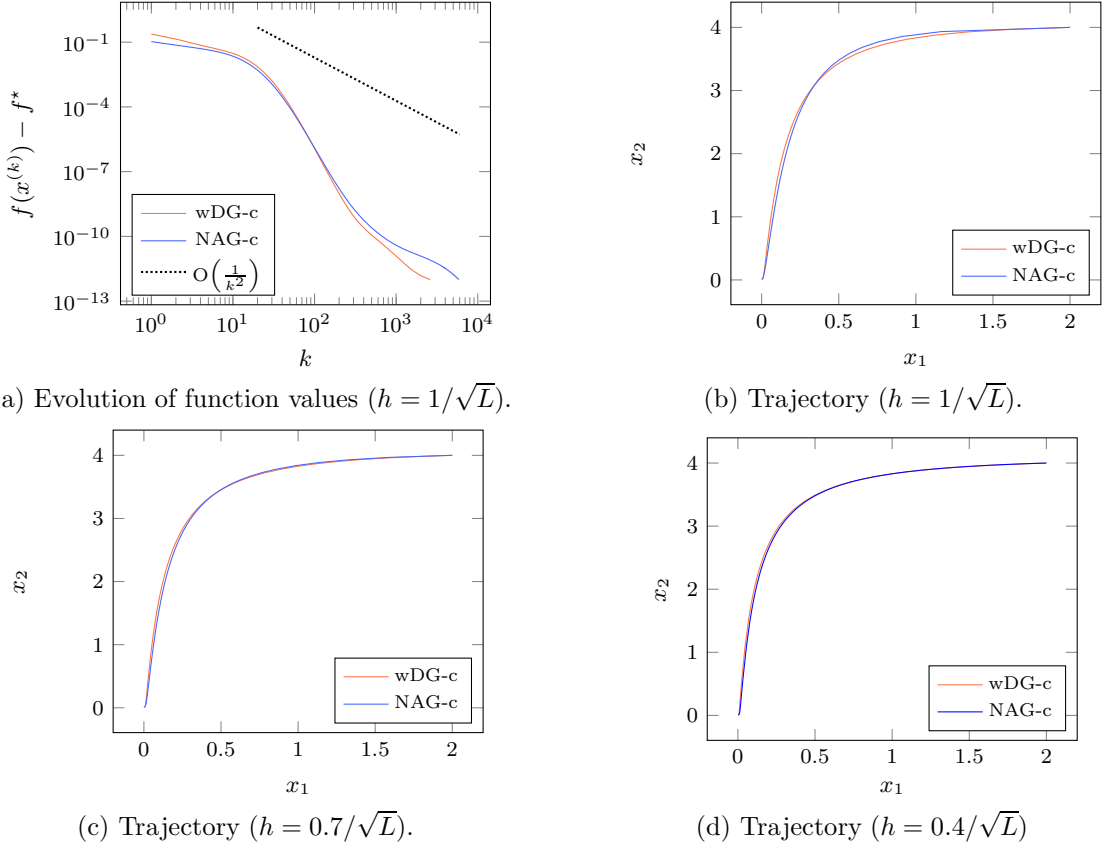
\begin{figure}[htbp]
  \begin{minipage}[b]{0.49\textwidth}
    \begin{center}
      \begin{tikzpicture}
        \begin{loglogaxis}[
          compat=1.16,
          legend pos= south west, 
          xlabel=$k$,ylabel=$f(\x{k})-f^\star$,
          small,no markers]
          \addplot[mark size=1pt, myorange, mark = +, mark repeat = 50, mark phase = 2] table [x index=0, y index=1] {dat/wDGctraj1.dat};
          \addplot+[mark size=1pt, myblue, mark = o, mark repeat = 50, mark phase = 2] table [x index=0, y index=1] {dat/NAGctraj1.dat};
          \addplot[black,densely dotted,thick,domain=20:6e3] {192/x/x};
          \legend{[font=\scriptsize]wDG-c,[font=\scriptsize]NAG-c,[font=\scriptsize] $ \Order{ \frac{1}{k^2}} $};
        \end{loglogaxis}
      \end{tikzpicture}%

      (a) Evolution of function values ($h = 1/\sqrt{L}$).
    \end{center}
  \end{minipage}
  \begin{minipage}[b]{0.48\textwidth}
    \begin{center}
      \begin{tikzpicture}
        \begin{axis}[
          legend pos= south east, 
          xlabel=$x_1$,ylabel=$x_2$,
          small,no markers]
          \addplot+[mark size=3pt, myorange, mark = +] table [x index=2, y index=3] {dat/wDGctraj1f100.dat};
          \addplot+[mark size=1pt, myblue, mark = o] table [x index=2, y index=3] {dat/NAGctraj1f100.dat};
          \legend{[font=\scriptsize]wDG-c,[font=\scriptsize]NAG-c};
        \end{axis}%
      \end{tikzpicture}%

      (b) Trajectory ($h = 1/\sqrt{L}$).
    \end{center}
  \end{minipage}

  \vspace{5pt}

  \begin{minipage}[b]{0.49\textwidth}
    \begin{center}
      \begin{tikzpicture}
        \begin{axis}[
          legend pos= south east, 
          xlabel=$x_1$,ylabel=$x_2$,
          small,no markers]
          \addplot+[mark size=3pt, myorange, mark = +] table [x index=2, y index=3] {dat/wDGctraj0.7f100.dat};
          \addplot+[mark size=1pt, myblue, mark = o] table [x index=2, y index=3] {dat/NAGctraj0.7f100.dat};
          \legend{[font=\scriptsize]wDG-c,[font=\scriptsize]NAG-c};
        \end{axis}%
      \end{tikzpicture}%

      (c) Trajectory ($h = 0.7/\sqrt{L}$).
    \end{center}
  \end{minipage}
  \begin{minipage}[b]{0.49\textwidth}
    \begin{center}
      \begin{tikzpicture}
        \begin{axis}[
          compat=1.16,
          legend pos= south east, 
          xlabel=$x_1$,ylabel=$x_2$,
          small,no markers]
          \addplot+[mark size=3pt, myorange, mark = +] table [x index=2, y index=3] {dat/wDGctraj0.4f100.dat};
          \addplot+[mark size=1pt, blue, mark = o] table [x index=2, y index=3] {dat/NAGctraj0.4f100.dat};
          \legend{[font=\scriptsize]wDG-c,[font=\scriptsize]NAG-c};
        \end{axis}%
      \end{tikzpicture}%

      (d) Trajectory ($h = 0.4/\sqrt{L}$)
    \end{center}
  \end{minipage}
  \caption{Trajectories and function values by (wDG-c) and (NAG-c). The objective function is~\eqref{obj5c} and the initial solution is $(2,4)$.}
  \label{fig:c}
\end{figure}

Next, we consider methods for strongly-convex functions.
We use the following explicit weak DG method:
\begin{equation}
  \simulparen{
      \xksxk &= \sqrt{\mu}h \paren*{ \vksxks },\\
      \vksvk &= \sqrt{\mu}h\paren*{\z{k} -\vv{k+1}-\frac{\nabla f \paren*{ \z{k} }}{\mu}}, \\
      \z{k} - \x{k} &= \sqrt{\mu}h \paren*{ \x{k} + \vv{k} - 2\z{k}}.
  }\label{5wDGsc}
\end{equation}
We call this (wDG-sc).
This can be obtained by setting $\WDG f(y,x) = \nabla f(x)$ in \cref{thm:cr_ac_sc_d}. 
Since for this choice we have $\alpha = L/2$ and $\beta = \mu/2$ (\cref{thm:wDG}), the step size condition is $h \le 1/(\sqrt{L} - \sqrt{\mu})$, and the predicted rate is $\Order{(1-\sqrt{\mu/L})^k}$ (which is attained by the largest $h = 1/(\sqrt{L} - \sqrt{\mu})$).

As before, we compare this method with Nesterov's accelerated gradient method for strongly convex functions (NAG-sc):
\begin{equation}
  \simulparen{
    \y{k+1} &= \x{k} - h^2 \nabla f \paren*{ \x{k} },\\
    \x{k+1} &= \y{k+1} + \frac{1-\sqrt{\mu}h}{1+\sqrt{\mu}h} \paren*{ \y{k+1}-\y{k} }.
  }\label{5NAGsc}
\end{equation}

Here, in addition to these, we also consider a simpler method, where $\z{k} = \x{k}$ is chosen to find
\begin{equation}
  \simulparen{
      \xksxk &= \sqrt{\mu}h \paren*{ \vksxks },\\
      \vksvk &= \sqrt{\mu}h\paren*{\x{k} -\vv{k+1}-\frac{\nabla f \paren*{ \x{k} }}{\mu}}. \\
  }\label{5wDGscnoy}
\end{equation}
We call this (wDG2-sc).
This method is more natural as a numerical method for the accelerated gradient flow~\eqref{accgfsc}, compared to the methods above, and we expect it illustrates how ``being natural as a numerical method'' affects the performance.
The rate and the step size limitation were revealed in the proof of~\cref{thm:cr_ac_sc_d} (\cref{app:proof_discrete}).

We summarized the step size limitations and rates in~\cref{table:5ex}.
Notice that the predicted rate of (wDG-sc) is better than that of (wDG2-sc).

The objective function is taken to be the quadratic function
\begin{equation}
  f(x) = 0.001(x_1-x_2)^2 + 0.1(x_1+x_2)^2 + 0.01x_1+0.02x_2, \label{obj5sc}
\end{equation}
and results are shown is \cref{fig:sc}.
Again the top-left panel shows the convergence of the objective function.
We see that (wDG-sc) and (NAG-sc) with each optimal step size show almost the same convergence,
which is in this case much better than the predicted worst case rate (the dotted guide line).
(wDG2-sc) slightly falls behind the other two, but it eventually achieves almost the same performance as $k\to\infty$.
The trajectories of the points $\x{k}$'s are, however, quite different among the three methods, which is interesting to observe.
The trajectory of (wDG2-sc) seems to suffer from wild oscillations, while (wDG-sc) generates milder trajectory.
(NAG-sc) comes between these two.
We need careful discussion to conclude which dynamics is the best as an optimization method, but if we consider such oscillations are not desirable (possibly causing some instability), it might suggest that (wDG-sc) is the first choice for this problem.
In any case, in this way we can explore various concrete optimization method within the framework of the weak DG by varying the weak DG, which is exactly the main claim of this paper.

\begin{figure}[htbp]
  \begin{minipage}[b]{0.49\textwidth}
    \begin{center}
      \begin{tikzpicture}
        \begin{semilogyaxis}[
          compat=1.16,
          legend pos= south west, 
          xlabel=$k$,ylabel=$f(\x{k})-f^\star$,
          small]
          \addplot+[mark size=3pt, myorange, mark = +] table [only marks, x index=0, y index=1] {dat/wDGsctraj1.dat};
          \addplot+[mark size=1pt, mylime, mark = x] table [only marks, x index=0, y index=1] {dat/wDGscnoytraj1.dat};
          \addplot+[mark size=1pt, myblue, mark = o] table [only marks, x index=0, y index=1] {dat/NAGsctraj1.dat};
          \addplot[black,densely dotted,thick,domain=1:140] {exp(x*ln(0.9))};
          \legend{[font=\scriptsize]wDG-sc,[font=\scriptsize]wDG2-sc,[font=\scriptsize]NAG-sc,[font=\scriptsize] $ \Order{ \paren*{ 1 - \sqrt{\frac{\mu}{L}} }^k } $};
        \end{semilogyaxis}%
      \end{tikzpicture}%

      (a) Evolution of function values ($h = 1/(\sqrt{L}-\sqrt{\mu})$ for wDG-sc and wDG2-sc, and $h = 1/\sqrt{L}$ for NAG-sc). 
    \end{center}
  \end{minipage}
  \begin{minipage}[b]{0.49\textwidth}
    \begin{center}
      \begin{tikzpicture}
        \begin{axis}[
          legend pos= north west, 
          xlabel=$x_1$,ylabel=$x_2$,
          small]
          \addplot+[mark size=3pt, myorange, mark = +] table [x index=2, y index=3] {dat/wDGsctraj1.dat};
          \addplot+[mark size=1pt, mylime, mark = x] table [x index=2, y index=3] {dat/wDGscnoytraj1.dat};
          \addplot+[mark size=1pt, myblue, mark = o] table [x index=2, y index=3] {dat/NAGsctraj1.dat};
          \legend{[font=\scriptsize]wDG-sc,[font=\scriptsize]wDG2-sc,[font=\scriptsize]NAG-sc};
        \end{axis}%
      \end{tikzpicture}%

      (b) Trajectory ($h = 1/(\sqrt{L}-\sqrt{\mu})$ for wDG-sc and wDG2-sc, and $h = 1/\sqrt{L}$ for NAG-sc).
    \end{center}
  \end{minipage}

\vspace{5pt}

  \begin{minipage}[b]{0.49\textwidth}
    \begin{center}
      \begin{tikzpicture}
        \begin{axis}[
          legend pos= north west, 
          xlabel=$x_1$,ylabel=$x_2$,
          small]
          \addplot+[mark size=3pt, myorange, mark = +] table [x index=2, y index=3] {dat/wDGsctraj0.7.dat};
          \addplot+[mark size=1pt, mylime, mark = x] table [x index=2, y index=3] {dat/wDGscnoytraj0.7.dat};
          \addplot+[mark size=1pt, myblue, mark = o] table [x index=2, y index=3] {dat/NAGsctraj0.7.dat};
          \legend{[font=\scriptsize]wDG-sc,[font=\scriptsize]wDG2-sc,[font=\scriptsize]NAG-sc};
        \end{axis}%
      \end{tikzpicture}%

      (c) Trajectory ($h = 0.7/\sqrt{L}$)
    \end{center}
  \end{minipage}
  \begin{minipage}[b]{0.49\textwidth}
    \begin{center}
      \begin{tikzpicture}
        \begin{axis}[
          compat=1.16,
          legend pos= north west, 
          xlabel=$x_1$,ylabel=$x_2$,
          small]
          \addplot+[mark size=3pt, myorange, mark = +] table [x index=2, y index=3] {dat/wDGsctraj0.4.dat};
          \addplot+[mark size=1pt, mylime, mark = x] table [x index=2, y index=3] {dat/wDGscnoytraj0.4.dat};
          \addplot+[mark size=1pt, myblue, mark = o] table [x index=2, y index=3] {dat/NAGsctraj0.4.dat};
          \legend{[font=\scriptsize]wDG-sc,[font=\scriptsize]wDG2-sc,[font=\scriptsize]NAG-sc};
        \end{axis}%
      \end{tikzpicture}%

      (d) Trajectory ($h = 0.4/\sqrt{L}$)
    \end{center}
  \end{minipage}
  \caption{Trajectories and function values by (wDG-sc), (wDG2-sc) and (NAG-sc). The objective function is~\eqref{obj5sc} and the initial value is $(2,3)$.}
  \label{fig:sc}
\end{figure}

\end{document}